\documentclass[a4paper,11pt]{amsart}

\usepackage[left=2.7cm,right=2.7cm,top=3.5cm,bottom=3cm]{geometry}
\usepackage{amsthm,amssymb,amsmath,amsfonts,mathrsfs,amscd}
\usepackage[latin1]{inputenc}
\usepackage[all]{xy}
\usepackage{latexsym}
\usepackage{longtable}

\theoremstyle{plain}
\newtheorem{theorem}{Theorem}[section]

\newtheorem{lemma}[theorem]{Lemma}
\newtheorem{proposition}[theorem]{Proposition}

\theoremstyle{definition}
\newtheorem{definition}[theorem]{Definition}

\newtheorem{examplewr}[theorem]{Examples}

\theoremstyle{remark}
\newtheorem{obswr}[theorem]{Observation}
\newtheorem{remarkwr}[theorem]{Remark}

\newcommand{\cF}{\mathcal F}
\newcommand{\cC}{\mathcal C}
\newcommand{\V}{\mathcal V}
\newcommand{\E}{\mathcal E}
\newcommand{\g}{\gamma}
\newcommand{\G}{\Gamma}
\newcommand{\Gap}{\Gamma_0^D(p M)}
\newcommand{\Ga}{\Gamma_0^D(M)}
\newcommand{\hGa}{\hat{\Gamma }_0(M)}
\newcommand{\Q}{\mathbb{Q}}
\newcommand{\Z}{\mathbb{Z}}

\newcommand{\C}{\mathbb{C}}

\newcommand{\PP}{\mathbb{P}}

\newcommand{\Gal}{\operatorname{Gal\,}}
\newcommand{\GL}{\operatorname{GL}}

\newcommand{\Div}{\operatorname{Div}}

\newcommand{\End}{\operatorname{End}}

\newcommand{\ord}{{\operatorname{ord}}}
\newfont{\gotip}{eufb10 at 12pt}

\newcommand{\cM}{{\mathcal M}}

\newcommand{\cL}{{\mathcal L}}

\newcommand{\cD}{{\mathcal D}}
\newcommand{\om}{{\omega}}

\newcommand{\lra}{\longrightarrow}

\newcommand{\SL}{{\mathrm {SL}}}

\newcommand{\R}{{\mathbb R}}
\newcommand{\M}{{\mathrm{M}}}
\newcommand{\HC}{\mathcal F_{\rm har}}
\newcommand{\PGL}{{\mathrm{PGL}}}
\newcommand{\Supp}{\mathrm{Supp}}

\def \mint {\times \hskip -1.1em \int}

\newcommand{\T}{\mathbb T}

\DeclareMathOperator{\Hom}{Hom}
\DeclareMathOperator{\Meas}{Meas}
\newcommand{\X}{\mathbb X}
\newcommand{\Y}{\mathbb Y}
\newcommand{\U}{\mathbb U}

\newcommand{\res}{\mathrm{res}}

\newcommand{\cl }{\mathcal}

\newcommand{\longmono}{\mbox{$\lhook\joinrel\longrightarrow$}}

\newcommand{\longepi}{\mbox{$\relbar\joinrel\twoheadrightarrow$}}

\newcommand{\smallmat}[4]{\bigl(\begin{smallmatrix}#1&#2\\#3&#4\end{smallmatrix}\bigr)}

\newcommand{\invlim}{\mathop{\varprojlim}\limits}

\include{thebibliography}

\begin{document}

\title[Uniformizations of Jacobians of Shimura curves]{On rigid analytic uniformizations of Jacobians of Shimura curves}

\author{Matteo Longo, Victor Rotger and Stefano Vigni}

\thanks{The research of the second author is financially supported by DGICYT Grant MTM2009-13060-C02-01.}

\begin{abstract}
The main goal of this article is to give an explicit rigid analytic uniformization of the maximal toric quotient of the Jacobian of a Shimura curve over $\Q$ at a prime dividing exactly the level. This result can be viewed as complementary to the classical theorem of \v{C}erednik and Drinfeld which provides rigid analytic uniformizations at primes dividing the discriminant. As a corollary, we offer a proof of a conjecture formulated by M. Greenberg in his paper on Stark--Heegner points and quaternionic Shimura curves, thus making Greenberg's construction of local points on elliptic curves over $\Q$ unconditional.
\end{abstract}

\address{M. L.: Dipartimento di Matematica, Università di Milano, Via C. Saldini 50, 20133 Milano, Italy}
\curraddr{Dipartimento di Matematica Pura e Applicata, Università di Padova, Via Trieste 63, 35121 Padova, Italy}
\email{mlongo@math.unipd.it}
\address{V. R.: Departament de Matemàtica Aplicada II, Universitat Politècnica de Catalunya, C. Jordi Girona 1-3, 08034 Barcelona, Spain}
\email{victor.rotger@upc.edu}
\address{S. V.: Dipartimento di Matematica, Università di Milano, Via C. Saldini 50, 20133 Milano, Italy}
\curraddr{Departament de Matemàtica Aplicada II, Universitat Politècnica de Catalunya, C. Jordi Girona 1-3, 08034 Barcelona, Spain}
\email{stefano.vigni@upc.edu}


\maketitle

\section{Introduction}

In an attempt to investigate analogues in the real setting of the theory of complex multiplication, Darmon introduced in his fundamental paper \cite{Dar} the notion of Stark--Heegner points on elliptic curves over $\Q$. These points are expected to be defined over abelian extensions of real quadratic fields $K$ (see \cite{BD2} for partial results in this  direction) and to satisfy analogous properties to those enjoyed by classical Heegner points rational over abelian extensions of imaginary quadratic fields.

Darmon's Stark--Heegner points were later lifted from elliptic curves to certain modular Jacobians by Dasgupta in \cite{Das}. More precisely, let $M$ be a positive integer and let $p$ be a prime number not dividing $M$. By working with modular symbols for the congruence subgroup $\Gamma_0(pM)$, Dasgupta defines a certain torus $T$ over $\Q_p$ and a lattice $L\subset T$, and he proves that the quotient $T/L$ is isogenous to the maximal toric quotient $J_0(p M)^{\text{$p$-new}}$ of the Jacobian of the modular curve $X_0(pM)$. This statement, which can be phrased as an equality of $\mathcal L$-invariants, turns out to be a strong form of the conjecture of Mazur--Tate--Teitelbaum (\cite{MTT}), now a theorem of Greenberg and Stevens (\cite{gs}). The very construction of $T/L$ allows Dasgupta to introduce Stark--Heegner points on it, and these points map to Darmon's ones under modular parametrizations. As a by-product, an efficient method for calculating the $p$-adic periods of $J_0(p M)^{\text{$p$-new}}$ is also obtained (in contrast with the less explicit approach of de Shalit in \cite{dS}).

It is important to observe that both Darmon's and Dasgupta's strategies, making extensive use of the theory of modular symbols, depend crucially on the presence of cusps on classical modular curves, and this prevents their arguments from extending in a straightforward way to the situation where modular curves are replaced by more general Shimura curves. In more explicit terms, the above methods apply only under the following \emph{Stark--Heegner hypothesis}:
\begin{equation} \label{SHhyp}
\text{$p$ is inert in $K$ and all the primes dividing $M$ split in $K$.}
\end{equation}
When $E$ is an elliptic curve over $\Q$ of conductor $N=pM$, condition \eqref{SHhyp} implies that the sign $\sigma(E/K)$ of the functional equation of the $L$-function attached to $E$ over $K$ is $-1$; the existence of Darmon's Stark--Heegner points is thus predicted by the conjecture of Birch and Swinnerton-Dyer for $E_{/K}$.

The starting point of our investigation is the recent article \cite{Gr} of M. Greenberg, in which the author proposes a program to generalize Darmon's constructions to totally real number fields and situations in which $\sigma(E/K)=-1$ but condition \eqref{SHhyp} is not satisfied. To give an idea of Greenberg's approach in the special case where the base field is $\Q$, assume that $K$ is a real quadratic field such that the conductor $N$ of the elliptic curve $E_{/\Q}$ admits a factorization $N=pDM$ into relatively prime integers, where $D$ is the square-free product of an \emph{even} number of primes, the prime divisors of $pD$ are inert in $K$ and the prime divisors of $M$ split in $K$. Then $\sigma(E/K)=-1$ and Greenberg describes a $p$-adic construction of Stark--Heegner points on $E$ which are conjectured to be rational over ring class fields of $K$ and to satisfy a suitable Shimura reciprocity law, as in the original work of Darmon. The key idea in \cite{Gr} is to reinterpret Darmon's theory of modular symbols in terms of the cohomology of Shimura curves attached to the quaternion algebra of discriminant $D$ and, ultimately, of group cohomology. Greenberg's construction of Stark--Heegner points on $E$ depends on the validity of an unproved statement (\cite[Conjecture 2]{Gr}) which is the counterpart of Dasgupta's version (\cite[Theorem 3.3]{Das}) of the theorem by Greenberg and Stevens; as a corollary to the main theorems in this paper, we give a proof of \cite[Conjecture 2]{Gr} over $\Q$, thus making Greenberg's results unconditional. This conjecture has also been proved, independently and by different methods, by Dasgupta and Greenberg in \cite{DG}.

More generally, the chief goal of our article is to give an explicit rigid analytic uniformization of the maximal toric quotient of the Jacobian of a Shimura curve associated with a non-split quaternion algebra at a prime dividing exactly the level, in the spirit of \cite{Das}. This result can be viewed as complementary to the classical theorem of \v{C}erednik and Drinfeld (for a detailed exposition of which we refer to \cite{BC}) that provides rigid analytic uniformizations at primes dividing the discriminant of the quaternion algebra. As will be made clear in the rest of the paper, our strategy is inspired by ideas in \cite{Das} and \cite{Gr}, to which we are indebted, and introduces several new ingredients for attacking the uniformization result; most remarkably, the explicit construction of a cocycle with values in a space of measures on $\PP^1(\Q_p)$ and an analysis of the delicate properties of a lift of it to a suitable bundle over $\PP^1(\Q_p)$, which are crucial for the proof of our main theorem. Beyond its theoretical interest, the construction of this cocycle is significant for a second reason: it is amenable to computations and -- with notation to be explained below -- paves the way to the calculation of the period matrix of $J^D_0(pM)^{\text{$p$-new}}$, as Dasgupta does in \cite[Section 6]{Das} for modular Jacobians.

Finally, we would like to highlight one more feature of our work. Although we devote no effort to this issue here, our results make it possible to define suitable lifts of Greenberg's Stark--Heegner points to Jacobians of Shimura curves, much in the same vein as the constructions in \cite{Das} lift Darmon's points to modular Jacobians. In fact, one of the long-run motivations of this article is to extend to broader contexts the results on the arithmetic of Stark--Heegner points, special values of $L$-functions and modular abelian varieties that are described by Bertolini, Darmon and Dasgupta in \cite{BDD}. Details in this direction will appear in future projects (see, e.g., \cite{LRV}); we hope that the results in the present paper may represent a first step towards a general and systematic study of special values of $L$-functions and congruences between modular forms over real quadratic fields as envisioned, for instance, in \cite[Section 6]{BD-2001} and \cite{BDD}.

Now let us describe the results of this paper more in detail; this will also give us the occasion to introduce some basic objects that will be used throughout our work. Let $D>1$ be a square-free product of an \emph{even} number of primes and let $M\geq1$ be an integer coprime with $D$. Let $B$ be the (unique, up to isomorphism) indefinite quaternion algebra over $\Q$ of discriminant $D$ and choose an isomorphism of algebras
\[ \iota_\infty:B\otimes_\Q\R\overset{\simeq}{\longrightarrow}\M_2(\R). \]
Let $R(M)$ be a fixed Eichler order of level $M$ in $B$ and write $\Gamma_0^D(M)$ for the group of norm $1$ elements in $R(M)$. Fix a prime $p\nmid MD$ and an Eichler order $R(pM)\subset R(M)$ of level $pM$ in $B$, and define as above $\Gamma_0^D(pM)$ to be the group of norm $1$ elements in $R(pM)$. Consider the compact Riemann surfaces
\begin{equation} \label{shimura-curves-eq}
X_0^D(M):=\Gamma_0^D(M)\backslash\cl H,\qquad X_0^D(pM):=\Gamma_0^D(pM)\backslash\cl H
\end{equation}
where $\cl H$ is the complex upper half-plane and the subgroup of elements in $B^\times$ with positive norm acts on $\cl H$ by M\"obius transformations via the embedding $B\hookrightarrow B\otimes_\Q\R$ and the isomorphism $\iota_\infty$. The curves in \eqref{shimura-curves-eq} are the \emph{Shimura curves} attached to $B$ of level $M$ and $pM$, respectively.

Let
\[ \pi_1,\pi_2: X^D_0(pM) \lra  X^D_0(M),\qquad \Gamma_0^D(pM)z \overset{\pi_1}{\longmapsto}\Gamma_0^D(M)z,\qquad \Gamma_0^D(pM)z \overset{\pi_2}{\longmapsto}\Gamma_0^D(M)\om_pz \]
be the two natural degeneracy maps; here $\om_p$ is an element in $R(pM)$ of reduced norm $p$ that normalizes $\Gamma_0^D(p M)$. Denote by $H$ the maximal torsion-free quotient of the cokernel of the map
\[ \pi^\ast:=\pi_1^\ast\oplus\pi_2^\ast:H_1\bigl(X_0^D(M),\Z\bigr)^2\longrightarrow H_1\bigl(X_0^D(pM),\Z\bigr) \]
induced by pull-back on homology. Let $J_0^D(pM)$ be the Jacobian variety of $X_0^D(pM)$ and let $J^D_0(pM)^{\text{$p$-new}}$ be its $p$-new quotient, whose dimension will be denoted by $g$; the abelian group $H$ is free of rank $2g$. Now consider the torus
\[ T:=\mathbb G_m\otimes_\Z H \]
where $\mathbb G_m$ denotes the multiplicative group (viewed as a functor on commutative $\Q$-algebras). Following the strategy of Dasgupta in \cite{Das}, we define a (full) lattice $L$ in $T$ and study the quotient $T/L$. In order to do this, fix an isomorphism of algebras
\begin{equation} \label{iota-p-eq}
\iota_p:B\otimes_\Q\Q_p\overset{\simeq}{\longrightarrow}\M_2(\Q_p)
\end{equation}
such that $\iota_p\big(R(M)\otimes\Z_p\big)$ is equal to $\M_2(\Z_p)$ and $\iota_p\big(R(pM)\otimes\Z_p\big)$ is equal to the subgroup of  $\M_2(\Z_p)$ consisting of upper triangular matrices modulo $p$. As done in \cite{Dar} when $D=1$, we introduce the group
\[ \Gamma:=\bigl(R(M)\otimes\Z[1/p]\bigr)^\times_1\;\overset{\iota_p}{\longmono}\;\GL_2(\Q_p), \]
which acts on Drinfeld's $p$-adic half-plane $\mathcal H_p:=\C_p-\Q_p$ with dense orbits. We will regard $H$ and $T$ as $\Gamma$-modules with trivial action.

In Sections \ref{section-Hecke} and \ref{L} we review some well-known facts on Hecke algebras, Shimura curves and $\mathcal L$-invariants. In Sections \ref{mvc}, \ref{raav} and \ref{periods} we introduce an explicit element in the cohomology group $H^1\bigl(\G,\Meas\bigl(\PP^1(\Q_p),H\bigr)\bigr)$ which defines by cup product an integration map on the homology group $H_1(\G,\Div^0\mathcal H_p)$ with values in $T(\C_p)$. We then consider the boundary homomorphism $H_2(\G,\Z)\rightarrow H_1(\G,\Div^0\mathcal H_p)$ induced by the degree map; the composition of these two maps produces a further map $H_2(\G,\Z)\rightarrow T(\C_p)$ whose image we denote by $L$. As we will see, it turns out that $L$ is a lattice of rank $2g$ in $T(\Q_p)$ which is preserved by the action of a suitable Hecke algebra. Finally, let $K_p$ denote the unramified quadratic extension of $\Q_p$.

The following is a precise formulation of the main result of this article, which is proved in Section \ref{uniformization-section}.

\begin{theorem} \label{GreenbergConj2}
The quotient $T/L$ admits a Hecke-equivariant isogeny over $K_p$ to the rigid analytic space associated with the product of two copies of $J_0^D(pM)^{\text{$p$-{\rm new}}}$.
\end{theorem}

We conclude this introduction by remarking that a proof of the conjecture proposed by M. Greenberg in \cite[Conjecture 2]{Gr} and alluded to above is given in \S \ref{greenberg-subsec}.

\bigskip

\noindent\emph{Notation and conventions.} If $M$ and $N$ are two abelian groups we write $M\otimes N$ for $M\otimes_\Z N$.

If $R$ is a ring and $M$ is a left $R$-module we endow $M$ with a structure of right $R$-module by the formula $m|r:=r^{-1}\cdot m$, where $(r,m)\mapsto r\cdot m$ is the structure map of $M$ as a left $R$-module.

For any ring $A$ and any $A$-module $M$ the symbol $M_T$ denotes the maximal torsion-free quotient of $M$.

For any discontinuous cocompact subgroup $G$ of $\mathrm{PSL}_2(\R)$ there are canonical isomorphisms
\[ H_1(G\backslash\mathcal H,\Z)\simeq H_1(G,\Z)_T\simeq G^{\rm ab}_T. \]
In the sequel we shall freely identify these three groups, and for every $g\in G$ we shall denote by $[g]\in G^{\rm ab}_T$ the class of $g$ in any of them. If $G_0$ is a subgroup of $G$ then the natural map $\pi:G_0\backslash\mathcal H\rightarrow G\backslash\mathcal H$ of Riemann surfaces induces by pull-back and push-forward homomorphisms
\[ \pi^*:H_1(G,\Z)_T\longrightarrow H_1(G_0,\Z)_T,\qquad\pi_*: H_1(G_0,\Z)_T\longrightarrow H_1(G,\Z)_T \]
which respectively translate, under the above identifications, to restriction and corestriction in homology of groups.

If $G$ is a group we denote by $(F_\bullet,\partial_\bullet)$ the standard resolution of $\Z$ by left $\Z[G]$-modules and, in non-homogeneous notations, we write $[g]=[g_1|\dots|g_r]$ for the elements of a $\Z[G]$-basis of $F_r$ as described in \cite[p. 18]{Br}.

For any right $\Z[G]$-module M we write, as usual, $B^r(G,M)\subset Z^r(G,M)\subset C^r(G,M)$ for the modules of $r$-coboundaries, $r$-cocycles and $r$-cochains, respectively, and $H^r(G,M):=Z^r(G,M)/B^r(G,M)$ for the $r$th cohomology group of $G$ with values in $M$. We use a similar notation, with lower indices this time, for homology.

We represent an element of $H_r(G,M)$ by an $r$-cycle $c=\sum_gm_g\otimes[g]$ in $Z_r(G,M)$. Likewise, we represent an element of $H^r(G,M)$ by a function $f:G^r\rightarrow M$ in $Z^r(G,M)$ and denote sometimes the value $f(g_1,\dots,g_r)$ by $f_{g_1,\dots,g_r}$. Finally, we adopt the description of boundary and coboundary maps given in \cite[p. 59]{Br}.

\vspace{0.5cm}

\noindent\emph{Acknowledgements.} We would like to thank the anonymous referee for the careful reading and for valuable remarks and suggestions. The second and third authors also thank the Centre de Recerca Matemàtica (Bellaterra, Spain) for its warm hospitality in Autumn 2009, when part of this research was carried out.

\section{Hecke operators on homology and cohomology} \label{section-Hecke}

\subsection{Review of the general theory} \label{section-formal-Hecke}

In this subsection we essentially follow \cite[\S 1.1]{AS}. Let $\mathcal G$ be a group; a \emph{Hecke pair (for $\mathcal G$)} consists of a subgroup $G$ of $\mathcal G$ and a subsemigroup $S$ of $\mathcal G$ such that
\begin{itemize}
\item $G\subset S$;
\item $G$ and $s^{-1}Gs$ are commensurable for every $s\in S$.
\end{itemize}
Now let $(G,S)$ be a Hecke pair and let $M$ be a left $\Z[S]$-module. Fix a double coset $GsG$ with $s\in S$ and form the finite disjoint decomposition $GsG=\coprod_is_iG$. Define the function $t_i:G\rightarrow G$ by the equation $g^{-1}s_i=s_{j(i)}t_i^{-1}(g)$. The Hecke operator $T(s)$ on the chain $c=\sum_gm_g\otimes[g]$ in $C_r(G,M)$ is defined by the formula
\[ T(s)\cdot c:=\sum_is_i^{-1}(m_g)\otimes\bigl[t_i(g_1)|\dots|t_i(g_r)\bigr], \]
where $g=[g_1|\dots|g_r]$ in non-homogeneous notation. Likewise, we define the Hecke operator $T(s)$ on the cochain $f\in C^r(G,M)$ by the formula
\[ T(s)\cdot f(g_1,\dots,g_r):=\sum_is_if\bigl(t_i(g_1),\dots,t_i(g_r)\bigr). \]
These operators induce operators, denoted by the same symbols, on $H_r(G,M)$ and $H^r(G,M)$.

If $M$ and $N$ are left $\Z[G]$-modules we may consider the cup product
\[ [\,,]:H_1(G,M)\times H^1(G,N)\longrightarrow H_0(G,M\otimes N), \]
which is defined as follows. Choose representatives $c=\sum_gm_g\otimes[g]$ of $\boldsymbol{c}\in H_1(G,M)$ and $f$ of $\boldsymbol{f}\in H^1(G,M)$;  then $[\boldsymbol{c},\boldsymbol{f}]$ is represented by
\[ [c,f]:=\sum_gm_g\otimes f(g). \]
It is easy to check that
\begin{equation} \label{Heckepairing}
[T(s)\cdot c,f]=[c,T(s)\cdot f],
\end{equation}
for all $s\in S$, from which one gets the equality
\[ \big[T(s)\cdot\boldsymbol{c},\boldsymbol{f}\big]=\big[\boldsymbol{c},T(s)\cdot\boldsymbol{f}\big]. \]

\subsection{Hecke algebras attached to Eichler orders over $\Z$} \label{HZ}

We apply the previous formal considerations to our arithmetic setting. Let $B$ be the quaternion algebra over $\Q$ of discriminant $D\geq1$; the requirement that $D>1$ will be made only from \S \ref{sec-def-meas} on. Fix an integer $N\geq1$ prime to $D$ and let $R\subset B$ be an Eichler order of level $N$; set $\Gamma_{R}:=R^\times$. For every integer $n\geq1$ and every prime $\ell$ define
\[ \Gamma_0^{\rm loc}(\ell^{n}):=\left\{\smallmat abcd\in\M_2(\Z_\ell)\mid c\equiv 0 \pmod{\ell^n}\right\}. \]
If $\ell\nmid N$ define $S_{R,\ell}$ to be the set of elements in $R\otimes\Z_\ell$ with non-zero norm. If there exists an integer $n_\ell\geq1$ such that $\ell^{n_\ell}|N$ and $\ell^{n_\ell+1}\nmid N$ fix an isomorphism of algebras
\[ \iota_\ell:B\otimes_\Q\Q_\ell\overset{\simeq}{\longrightarrow}\M_2(\Q_\ell) \]
such that $\iota_\ell(R\otimes\Z_\ell)=\Gamma_0^{\rm loc}(\ell^{n_\ell})$, and define $S_{R,\ell}$ to be the inverse image of the semigroup consisting of matrices $\smallmat abcd\in\M_2(\Z_\ell)$ with $c\equiv0\pmod{\ell^{n_\ell}}$, $a\in\Z_\ell^\times$ and $ad-bc\neq0$. Finally, set
\[ S_R:=B^\times\cap\prod_\ell S_{R,\ell}, \]
where the product is taken over all prime numbers $\ell$. Then $(\Gamma_R,S_R)$ is a Hecke pair. Write ${\rm nr}:B\rightarrow\Q$ for the reduced norm; for every integer $n\geq1$ define
\[ T_n:=\sum_{\substack{\alpha\in S_R\\{\rm nr}(\alpha)=n}}T(\alpha), \]
and for every integers $n\geq1$ prime to $ND$ define
\[ T_{n,n}:=T(n). \]
If $\ell$ is a prime then we have $T_\ell=T(g_0)$ for a certain $g_0=g_0(\ell)\in R$; moreover,
\[ \G_Rg_0\G_R=\coprod_ig_i\G_R \]
for some $g_i=g_i(\ell)\in R$ of reduced norm $\ell$ and $i\in \{0,\dots,\ell\}$ if $\ell\nmid N$ (respectively, $i\in\{0,\dots,\ell-1\}$ if $\ell|N$). As customary, if $\ell|N$ is a prime we will also denote $T_\ell$ by $U_\ell$ to emphasize that we are considering an operator at a prime dividing the level. The Hecke algebra $\cl H(\Gamma_R,S_R)$ of the pair $(\Gamma_R,S_R)$, defined in \cite[p. 194]{AS}, is commutative and can be explicitly described as
\[ \cl H(\Gamma_R,S_R)=\Z\bigl[\text{$T_\ell$ for all primes $\ell$, $T_{\ell,\ell}$ for primes $\ell\nmid ND$}\bigr]. \]
See \cite[\S 5.3]{Miy} for details and proofs. As before, let $\om_p\in R(pM)$ be a fixed element of reduced norm $p$ which normalizes $\Gap$; as a piece of notation, for $R=R(M)$, $R=R(pM)$ and $R=\hat R(pM):=\om_p R(pM)\om_p^{-1}$ we denote $\cl H (\Gamma_R,S_R)$ by $\cl H(M)$, $\cl H(pM)$ and $\hat{\mathcal H}(pM)$, respectively.

We will be particularly interested in the Hecke operator $U_p\in\mathcal H(pM)$. In this case, $U_p = T(g_0)$ for a fixed choice of $g_0\in R(pM)$ of reduced norm $p$ such that $\iota_p(g_0)=\smallmat100p u_0$ for some $u_0\in\Gamma_0^{\rm loc}(p)$; the element $g_0$ gives rise to a coset decomposition
\[ \Gamma_0^D(pM)g_0\Gamma_0^D(pM)=\coprod_{i=0}^{p-1}g_i\Gamma_0^D(pM) \]
with the $g_i$ such that $\iota_p(g_i)=\smallmat 10{pi}pu_i$ for some $u_i\in\Gamma_0^{\rm loc}(p)$ and every $i\in\{0,\dots,p-1\}$.

Fix once and for all an element $\om_{\infty}\in R(pM)$ of reduced norm $-1$ which normalizes $\Gap$. In addition to the operators described above, the involutions $W_p=T(\om_p)$ and $W_\infty=T(\om_\infty)$ in $\mathcal H(pM)$ will also play a key role in our discussion. More precisely, $\om_p$ can be taken such that
\[ \iota_p(\om_p)\in \smallmat0{-1}p0\cdot\Gamma_0^{\rm loc}(p). \]
A direct computation then shows that the $\alpha_i:=\om_p^{-1}g_i$ lie in $\Ga$ and that, actually, $\{\alpha_\infty=1,\alpha_0,\dots,\alpha_{p-1}\}$ is a set of representatives of $\Ga/\Gap$; from this one deduces the well-known fact that $U_p=-W_p$ on $H$.

\subsection{A Hecke algebra attached to $R(M)\otimes\Z[\frac{1}{p}]$}

The formalism described in \S \ref{section-formal-Hecke} can also be applied to the Hecke pair $(\Gamma, S_{(p,M)})$ where $\G$ is as in the introduction and
\[ S_{(p,M)}:=B^\times\cap\prod_{\ell\not=p} S_{R(M),\ell}, \]
the product being taken over all prime numbers different from $p$. Throughout we shall write $\cl H(p,M)$ as a shorthand for the Hecke algebra corresponding to the pair $(\Gamma, S_{(p,M)})$, which is again commutative.

Similarly as before, in this algebra one defines Hecke operators $T_\ell$ for primes $\ell\ne p$ and involutions $W_p$ and $W_\infty$. These operators correspond to double coset decompositions $\G g_0(\ell)\G=\coprod_ig_i(\ell)\G$, $\G\om_p\G=\om_p\G$ and $\G\om_\infty\G=\om_\infty\G$, respectively, with exactly the \emph{same choices} of the $g_i(\ell)$, of $\om_p$ and of $\om_\infty$ in $R(pM)$ as made before. Finally, in complete analogy with what has just been done, one can also introduce the Hecke operator $U_p$. However, since $\G g\G=\om_p\G$ for any $g\in S_{(p,M)}$ of reduced norm $p$, now we have $U_p=W_p$.

\section{$\mathcal L$-invariants}\label{L}

\subsection{Singular homology groups}

Recall the isomorphism $\iota_p:B\otimes_\Q\Q_p\simeq\M_2(\Q_p)$ of \eqref{iota-p-eq} and the Eichler order $R=R(M)$ of $B$ of level $M$ chosen in the introduction. For every integer $r\geq1$ let $C_r\subset\widehat R^\times$ denote the subgroup of elements whose $p$-component is mapped by $\iota_p$ to a matrix $\smallmat abcd$ such that $a\equiv 1\pmod{p^r}$ and $c\equiv0\pmod{p^r}$. Moreover, let $\Gamma_r^D$ be the subgroup of norm $1$ elements in $C_r\cap B^\times$; we shall write $\G^D_r(M)$ whenever the level of $R$ is not fixed from the outset.

Note that $\Gamma_r^1= \Gamma_0(M)\cap\Gamma_1(p^r)$
as a congruence subgroup of $\SL_2(\Z)$. Finally, in order to have uniform notations, set
$X_0^D:=X_0^D(pM)$ and $\Gamma_0^D:=\Gamma_0^D(pM)$.

For every $r\geq1$ define the compact Riemann surfaces
\[ X_r^D:=\Gamma_r\backslash\cl H\quad\text{if $D>1$},\qquad X_r^1:=\Gamma_r\backslash\cl H\cup\{\text{cusps}\}, \]
where $X^1_r$ is the compactification of the open modular curve $Y^1_r:=\Gamma_r\backslash\cl H$ obtained by adding a finite number of cusps.

Let $S_2(\Gamma_r^D,\C)$ be the $\C$-vector space of holomorphic $1$-forms on $X_r^D$, which is isomorphic to $H^1 (X_r^D,\R)$ as an $\R$-vector space (see, e.g., \cite[Theorem 8.4]{Sh2}). In particular, if $g_r^D$ is the genus of $X_r^D$ then the dimension of $H^1(X_r^D,\R)$ over $\R$ is $2g_r^D$. Since $X_r^D$ is compact, Poincar{\'e} duality gives an isomorphism
\[ H^1\bigl(X_r^D,\R\bigr)\simeq H_1\bigl(X_r^D,\R\bigr) \]
of $\R$-vector spaces. As a consequence, the universal coefficient theorem for homology yields canonical isomorphisms of $\R$-vector spaces
\begin{equation} \label{eq-1}
S_2\bigl(\Gamma_r^D,\C\bigr)\simeq H_1\bigl(X_r^D,\R\bigr)\simeq H_1\bigl(X_r^D,\Z\bigr)\otimes_\Z\R
\end{equation}
In particular, the abelian group $H_1(X_r^D,\Z)$ is free of rank $2g_r^D$. The above discussion and the universal coefficient theorem for homology show that $H_1(X_r^D,\Z_p)$ is also free of rank $2g_r^D$ as a $\Z_p$-module.

There are canonical projection maps
\[ \pi_{1,r}:X_r^D\longrightarrow X_0^D(M),\qquad X_r^D\longrightarrow X_s^D \]
for $r\geq s\geq0$. For every integer $r\geq0$ let $W_p$ denote the Atkin--Lehner involution on $X_r^D$ defined as in \cite[\S 1.5]{BD0} via the adelic description of $X_r^D$ as the double coset space
\[ X_r^D=B^\times\big\backslash\bigl(\widehat B^\times\times\cl H\bigr)\big/C_r. \]
Explicitly, $W_p$ is the map $[(g,z)]\mapsto\bigl[\bigl(\smallmat0{-1}{{p}}0\cdot g,z\bigr)\bigr]$. Define
\[ \pi_{2,r}:=\pi_{1,r}\circ W_p:X_r^D\longrightarrow X_0^D(M). \]
This gives rise to a map
\[ \pi_r:=\pi_{1,r}\times \pi_{2,r}:X_r^D\longrightarrow X_0^D(M)\times X_0^D(M) \]
and thus, by pull-back, to a map
\[ \pi_r^*:H_1\bigl(X_0^D(M),\Z_p\bigr)\oplus H_1\bigl(X_0^D(M),\Z_p\bigr)\longrightarrow H_1\bigl(X_r^D,\Z_p\bigr) \]
on homology groups. (Note that for $r=0$ these maps coincide with the maps $\pi_1$, $\pi_2$, $\pi$, $\pi_1^*$, $\pi_2^*$, $\pi^*$ appearing in the introduction.) For $r\geq 0$ define the $\Z_p$-module
\[ H_r^D:=\bigl[H_1\bigl(X_r^D,\Z_p\bigr)/\text{Im}(\pi_r^*)\bigr]_T \]
and let $\T_r^D$ denote the image in $\End_{\Z_p}(H_r^D)$ of the Hecke algebra $\cl H(pM)\otimes\Z_p$; as above, we shall rather write $\T_r^D(M)$ if needed. Thanks to isomorphisms \eqref{eq-1}, it follows that $\T_r^D$ is canonically identified with the $p$-new quotient of the classical Hecke algebra acting on $S_2\bigl(\Gamma_r^D,\C\bigr)$ as defined, for example, in \cite[\S 2]{Hida-annals}.

\subsection{Jacquet--Langlands correspondence} \label{section-JL}

Denote by $\T_r^{1,\text{$D$-new}}(D M)$ the quotient of the Hecke algebra $\T^1_r(D M)$ acting faithfully on the $\C$-vector space of weight $2$ cusp forms on $\Gamma^1_r(DM)$ which are new at $D$. For every $r\geq0$ the Hecke-equivariance of the Jacquet--Langlands correspondence between classical and quaternionic modular forms (see, e.g., \cite[Theorem 2.30]{Hida-HMFIT}) gives a canonical isomorphism of rings
\begin{equation} \label{JL-eq}
{\rm JL}_r: \T_r^{1,\text{$D$-new}}(D M) \overset{\simeq}{\longrightarrow}\T_r^D(M)
\end{equation}
making the natural diagram
\[ \xymatrix@C=30pt@R=25pt{\T_r^{1,\text{$D$-new}}(D M)\ar[r]^-{{\rm JL}_r}\ar[d] & \T_r^D(M)\ar[d] \\
                           \T_s^{1,\text{$D$-new}}(D M)\ar[r]^-{{\rm JL}_s} & \T_s^D(M)} \]
commutative.

\subsection{Quaternionic Hida theory}

Being finitely generated as a $\Z_p$-module, the algebra $\T_r^D(M)$ is isomorphic to the product of the localizations at its (finitely many) maximal ideals; write $\T_r^{D,\ord}=\T_r^{D,\ord}(M)$ for the product of those local components in which the image of $U_p$ is a unit. Following Hida, define the \emph{ordinary Hecke algebra} as
\[ \T_\infty^{D,\ord}:=\invlim_r\T_r^{D,\ord}\qquad(r\geq1). \]
Furthermore, if we set
\[ \T_\infty^{1,\text{$D$-new}}(DM):=\invlim_r\T_r^{1,\text{$D$-new}}(DM) \]
then isomorphism \eqref{JL-eq} shows that there is a canonical isomorphism
\[ \T_\infty^{1,\text{$D$-new},\ord}(DM)\simeq\T_\infty^{D,\ord}(M). \]
Denote by
\[ \Lambda:=\Z_p[\![1+p\Z_p]\!],\qquad\tilde\Lambda:=\Z_p[\![\Z_p^\times]\!] \]
the Iwasawa algebras of $1+p\Z_p$ and $\Z_p^\times$, respectively, so that $\tilde\Lambda$ has a natural $\Lambda$-algebra structure. There is a structure of $\tilde\Lambda$-module on $\T_\infty^{D,\ord}(M)$ defined on group-like elements $d\in\Z_p^\times$ by $d\mapsto\langle d\rangle$. Since, as a consequence of the Jacquet--Langlands isomorphism, $\T_r^D(M)$ is a quotient of $\T_r^{1}(DM)$ and the projection map takes $U_p$ to $U_p$, there is a canonical surjective map of $\tilde\Lambda$-algebras
\[ \T_\infty^{1,\ord}(D M)\;\longepi\;\T_\infty^{D,\ord}(M). \]
Thanks to \cite[Theorem 3.1]{h-iwasawa}, $\T_\infty^{1,\ord}(DM)$ is a $\Lambda$-algebra which is free of finite rank as a $\Lambda$-module. In particular, it immediately follows that $\T_\infty^{D,\ord}(M)$ is finitely generated as a $\Lambda$-module. Thanks to \cite[Corollary 3.2]{h-iwasawa} (see also \cite[Theorem 5.6]{Das} for the result in this form), if $I_{\tilde\Lambda}$ is the augmentation ideal of $\tilde\Lambda$ then the canonical projection
\[ \T_\infty^{1,\ord}(DM)\;\longepi\;\T^{1,\ord}_0(DM) \]
induces an isomorphism of $\Z_p$-algebras
\begin{equation} \label{thm-hida}
\rho:\T_\infty^{1,\ord}(DM)\big/I_{\tilde\Lambda}\T_\infty^{1,\ord}(DM)\overset{\simeq}{\longrightarrow}\T_0^{1,\ord}(DM).
\end{equation}
The next result is the counterpart of isomorphism \eqref{thm-hida} in our general quaternionic setting.
\begin{proposition} \label{prop-hida}
For every $D\geq1$ the canonical projection $\T_\infty^{D,\ord}(M)\twoheadrightarrow\T^{D,\ord}_0(M)$ induces an isomorphism of $\Z_p$-algebras
\[ \rho_D:\T_\infty^{D,\ord}(M) \big/I_{\tilde\Lambda}\T_\infty^{D,\ord}(M)\overset{\simeq}{\longrightarrow}\T_0^{D,\ord}(M) \]
which sits in the commutative diagram
\[ \xymatrix{\T_\infty^{1,\ord}(DM)\big/I_{\tilde\Lambda}\T_\infty^{1,\ord}(DM)\ar[r]^-{\rho_1}\ar@{->>}[d] & \T_0^{1,\ord}(DM)\ar@{->>}[d]\\
             \T_\infty^{D,\ord}(M)\big/I_{\tilde\Lambda}\T_\infty^{D,\ord}(M)\ar[r]^-{\rho_D} & \T_0^{D,\ord}(M)} \]
where the vertical arrows are the canonical surjections.
\end{proposition}

\begin{proof} For $D=1$ this is simply \eqref{thm-hida}. In general, we only have to show that the kernel of the canonical projection
\begin{equation} \label{rho-D-eq}
\T_\infty^{D,\ord}(M)\;\longepi\;\T^{D,\ord}_0(M)
\end{equation}
is $I_{\tilde\Lambda}\T_\infty^{D,\ord}(M)$. It is straightforward to check that $I_{\tilde\Lambda}\T_\infty^{D,\ord}(M)$ is indeed contained in the kernel of the homomorphism in \eqref{rho-D-eq}, hence there is a surjection
\[ \rho_D:\T_\infty^{D,\ord}(M)\big/I_{\tilde\Lambda}\T_\infty^{D,\ord}(M)\;\longepi\;\T_0^{D,\ord}(M) \]
of $\Z_p$-algebras. For every integer $r\geq0$ let us denote by $\T_r^{1,\text{$D$-old},\ord}(DM)$ the kernel of the projection $\T_r^{1,\ord} (DM)\twoheadrightarrow\T_r^{D,\ord}(M)$ induced by the Jacquet--Langlands correspondence recalled in \S \ref{section-JL}, so that we have a canonical short exact sequence
\begin{equation} \label{short-exact-eq}
0\longrightarrow\T_r^{1,\text{$D$-old},\ord}(DM)\longrightarrow\T_r^{1,\ord}(DM)\longrightarrow\T_r^{D,\ord}(M)\longrightarrow0.
\end{equation}
After setting $\T_\infty^{1,\text{$D$-old},\ord}(DM):=\invlim\T_r^{1,\text{$D$-old},\ord}(DM)$ and tensoring by $\tilde\Lambda/I_{\tilde\Lambda}$ over $\tilde\Lambda$, from sequence \eqref{short-exact-eq} we obtain the diagram
\[ \footnotesize{\xymatrix@C=15pt{&\T_\infty^{1,\text{$D$-old},\ord}(DM)\big/I_{\tilde\Lambda}\T_\infty^{1,\text{$D$-old},\ord}
                 \ar[r]\ar@{->>}[d]&\T_\infty^{1,\ord}(DM)\big/I_{\tilde\Lambda}\T_\infty^{1,\ord}(DM)\ar[r]\ar[d]^-{\rho}_-{\simeq}
                 &\T_\infty^{D,\ord}(M)\big/I_{\tilde\Lambda}\T_\infty^{D,\ord}(M)\ar[r]\ar@{->>}[d]^-{\rho_D}&0\\
                 0\ar[r]&\T_0^{1,\text{$D$-old},\ord}(DM)\ar[r]&\T_0^{1,\ord}(DM)\ar[r]&\T_0^{D,\ord}(M)\ar[r]&0}} \]
with exact rows and surjective left vertical arrow. The snake lemma then implies that the kernel of $\rho_D$ is trivial, which shows that $\rho_D$ is an isomorphism. \end{proof}

\subsection{Definition of the $\mathcal L$-invariant}\label{sec-expl-spl}

The map $[d-1]\mapsto d$ yields a canonical identification
\[ I_{\tilde\Lambda}/I_{\tilde\Lambda}^2\overset{\simeq}{\longrightarrow}\Z_p^\times. \]
Composing with the branch $\log_p:\Z_p^\times\rightarrow\Z_p$ of the $p$-adic logarithm satisfying $\log_p(p)=0$ we then obtain a map
\[ \log_p:I_{\tilde\Lambda}/I_{\tilde\Lambda}^2\longrightarrow\Z_p \]
which, by a notational abuse, will be denoted by the same symbol. The composition of the isomorphism
\[  I_{\tilde\Lambda}\T_\infty^{D,\ord}(M)\big/I_{\tilde\Lambda}^2\T_\infty^{D,\ord}(M)\simeq\T_\infty^{D,\ord}(M)\otimes_{\tilde\Lambda}\bigl(I_{\tilde\Lambda}\big/I_{\tilde\Lambda}^2\bigr) \]
with the map
\[ {\rm id}\otimes\log_p:\T_\infty^{D,\ord}(M)\otimes_{\tilde\Lambda}\bigl(I_{\tilde\Lambda}/I_{\tilde\Lambda}^2\bigr)\longrightarrow\T_\infty^{D,\ord}(M)\otimes_{\tilde\Lambda}\Z_p \]
produces a map
\begin{equation} \label{eq-hida-1}
I_{\tilde\Lambda}\T_\infty^{D,\ord}(M)\big/I_{\tilde\Lambda}^2\T_\infty^{D,\ord}(M)\longrightarrow\T_\infty^{D,\ord}(M)\otimes_{\tilde\Lambda}\Z_p.
\end{equation}
Now note that
\[ \T_\infty^{D,\ord}(M)\otimes_{\tilde\Lambda}\Z_p\simeq\T_\infty^{D,\ord}(M)\otimes_{\tilde\Lambda}\bigl(\tilde\Lambda/I_{\tilde\Lambda}\bigr)\simeq \T_\infty^{D,\ord}(M)\big/I_{\tilde\Lambda}\T_\infty^{D,\ord}(M)\simeq\T^{D,\ord}_0(M), \]
the last isomorphism following from Proposition \ref{prop-hida}. Composing this chain of isomorphisms with \eqref{eq-hida-1} yields a map
\[ I_{\tilde\Lambda}\T_\infty^{D,\ord}(M)\big/I_{\tilde\Lambda}^2\T_\infty^{D,\ord}(M)\longrightarrow\T^{D,\ord}_0(M). \]
Finally, composing with the canonical projection $I_{\tilde\Lambda}\T_\infty^{D,\ord}(M)\twoheadrightarrow I_{\tilde\Lambda}\T_\infty^{D,\ord}(M)\big/I_{\tilde\Lambda}^2\T_\infty^{D,\ord}(M)$ we obtain a map
\begin{equation} \label{eq-hida-2}
I_{\tilde\Lambda}\T_\infty^{D,\ord}(M)\longrightarrow\T^{D,\ord}_0(M)
\end{equation}
which is denoted by $t\mapsto t'$.

As discussed in \S \ref{HZ}, $U_p+W_p=0$ on $H_0^D$; hence, since $W_p^2=1$, we conclude that the image of $1-U_p^2$ in $\T_0^{D,\ord}(M)$ is trivial. It follows that
\[ 1-U_p^2\in I_{\tilde\Lambda}\T_\infty^{D,\ord}(M). \]

\begin{definition}
The \emph{$\mathcal L$-invariant}
\[ {\cl L}^D_p={\cl L}^D_p(M):=\bigl(1-U_p^2\bigr)'\in\T_0^{D,\ord}(M) \]
is the image of $1-U_p^2$ under the map \eqref{eq-hida-2}.
\end{definition}

Observe that $\cl L_p^1$ is equal to the $\mathcal L$-invariant defined by Dasgupta in \cite[Definition 5.2]{Das}.

\begin{proposition}
The $\mathcal L$-invariant ${\cl L}^D_p(M)$ is the image of $\cl L_p^1(DM)$ under the canonical surjection $\T_0^{1,\ord}(DM)\twoheadrightarrow\T_0^{D,\ord}(M)$.
\end{proposition}

\begin{proof} This follows immediately from the definition of the $\mathcal L$-invariants and the commutativity of the diagram in Proposition \ref{prop-hida}. \end{proof}

\subsection{Singular points and $\mathcal L$-invariants} \label{singular-subsec}

The arguments in this subsection are essentially a formal variation on those in \cite{gs} and \cite[\S 5.4]{Das}, so we will be rather sketchy.

Let $X:=\Div^0(\mathcal S)$ denote the group of degree zero divisors on the set $\mathcal S$ of supersingular points of $X_0^D(M)$ in characteristic $p$, and write $X^*:=\mathrm{Hom}_\Z(X,\Z)$ for its $\Z$-dual. As explained, e.g., in \cite[\S 1.7]{helm}, the group $X$ has a natural Hecke action; moreover, the Hecke algebra of $X$ canonically identifies with that of $H_0^D$. There is a non-degenerate, symmetric pairing
\[ Q: X\times X\longrightarrow\Q^\times_p \]
for which the Hecke operators are self-adjoint. The map $Q$ defines an injection
\[ j:X\;\longmono\;X^*\otimes\Q_p^\times \]
by setting $j(x)(y):=Q(x,y)$. For simplicity, put $G_p:=\Gal(\bar\Q_p/\Q_p)$ for the local Galois group at $p$; there is a short exact sequence
\begin{equation} \label{unif-I}
0\longrightarrow X\stackrel{j}{\longrightarrow}X^*\otimes\bar\Q_p^\times\longrightarrow J_0^D(pM)^{p{\text{-new}}}(\bar\Q_p)\longrightarrow0
\end{equation}
of left $\T_0^D(M)[G_p]$-modules. Composing the pairing $Q$ with the $p$-adic valuation $\ord_p$ gives rise to the non-degenerate monodromy pairing
\[ \ord_p\circ Q:X\times X\rightarrow \Z \]
at $p$. Now set $\ord_X(x)(y):=\ord_p\bigl(Q(x,y)\bigr)$, thus obtaining an injection
\[ \ord_X:X\;\longmono\;X^*. \]
Analogously, if $\log_p$ is the branch of the $p$-adic logarithm such that $\log_p(p)=0$ then we obtain a map
\[ \log_X:X\longrightarrow X^*\otimes_\Z\Z_p \]
defined by $\log_X(x)(y):=\log_p\bigl(Q(x,y)\bigr)$.

Recall that $\mathcal L_p^D:=\mathcal L_p^D(M)$; the next result seems to be well known to experts.

\begin{proposition} \label{prop-II}
There is an equality $\mathcal L_p^D\cdot\ord_X=\log_X$ of maps from $X$ to $X^\ast\otimes\Z_p$.
\end{proposition}

There are at least two ways of proving the above statement but, for the sake of brevity, we will not provide any details, as the methods are very similar to the standard ones in the classical modular setting, already present in the literature (\cite{Das}, \cite{gs}). One way of showing it is to proceed as in the proof of \cite[Proposition 5.20]{Das}, upon noticing that the arguments of \cite[\S 8]{MW} can be adapted to our quaternionic setting. Besides, more indirectly, one can also prove Proposition \ref{prop-II} by exploiting the commutativity of the diagram of Proposition \ref{prop-hida} combined with \cite[Proposition 5.20]{Das}.

\section{Measure-valued 1-cocycles}\label{mvc}

\subsection{Bruhat--Tits tree, harmonic cocycles and measures on $\PP^1(\Q_p)$} \label{BT}

Let $\cl T$ be the Bruhat--Tits tree of $\M_2(\Q_p)$, whose set $\V=\V(\cl T)$ of vertices consists of the maximal orders of $\M_2(\Q_p)$. We denote by $v_\ast$ the vertex $\M_2(\Z_p)$ and by $\hat{v}_\ast$ the vertex $\bigl\{\smallmat a{p^{-1}b}{pc}d\mid a,b,c,d\in\Z_p\bigr\}$.

The set $\E =\E(\cl T)$ of oriented edges of $\cl T$ is the set of ordered pairs $(v_1, v_2)$ with $v_1,v_2\in\V$ such that $v_1\cap v_2$ is an Eichler order of level $p$. We call $v_1=s(e)$ and $v_2=t(e)$ the source and the target of $e$, respectively, and write $\bar e$ for the reversed edge $(v_2,v_1)$. Set $e_\ast:=(v_\ast, \hat{v}_\ast)$.

Given $v,v'\in\V$, the distance between $v$ and $v'$ is the length of a path without backtracking from $v$ to $v'$, i.e., the smallest number of edges needed to connect $v$ with $v'$.

The group $\GL_2(\Q_p)$ acts transitively and isometrically on $\cl V$ by the rule $v\mapsto g v g^{-1}$ for $v\in\cl V$ and $g\in\GL_2(\Q_p)$. Hence, it also gives rise to a natural action of $\GL_2(\Q_p)$ on $\cl E$, which is again transitive. As a piece of notation, write $\hat{v}:=\om_p(v)$ and $\hat{e}:=\om_p(e)$ for any $v\in \V$ and any $e\in \E$, respectively. Similarly, for any $\g\in\GL_2(\Q_p)$ and any subgroup $G$ of $\GL_2(\Q_p)$ write $\hat{\g}:=\om_p\g\om_p^{-1}$ and $\hat{G}:=\om_p G\om_p^{-1}$. Observe that $\hat{e}_\ast=\bar{e}_\ast$ for all $e\in\E$.

We say that a vertex of $\cl T$ is \emph{even} (respectively, \emph{odd}) if its distance from  $v_\ast$ is even (respectively, odd). We write $\V^+$ (respectively, $\V^-$) for the subset of $\V$ consisting of even (respectively, odd) vertices, and we write $\E^+$ (respectively, $\E^-$) for the subset of $\E$ made up of those oriented edges, called \emph{even} (respectively, \emph{odd}), whose source is even (respectively, odd). Notice that $\V^-=\hat{\V}^+$ and $\E^- = \bar{\E}^+=\hat{\E}^+$.

Let $\GL_2^+(\Q_p)$ be the subgroup of $\GL_2(\Q_p)$ whose elements are the matrices $\g$ such that $\ord_p(\det (\g))$ is even, and recall from the introduction the subgroup
\[ \Gamma:=\bigl(R(M)\otimes_\Z \Z[1/p]\bigr)^\times_1\;\overset{\iota_p}{\longmono}\;\GL^+_2(\Q_p). \]
It follows from \cite[Ch. II, Theorem 2]{Se} that the segment connecting $v_\ast$ and $\hat{v}_\ast$ is a fundamental domain for the action of $\Gamma$ on $\cl T$, by which we mean a subgraph $\cl T'$ of $\cl T$ such that every vertex (respectively, edge) of $\cl T$ is $\Gamma$-equivalent to a vertex (respectively, edge) in $\cl T'$. The stabilizers of $v_\ast$, $\hat{v}_\ast$ and $e_\ast$ in $\G$ are $\Ga$, $\hat{\Gamma}^D_0(M)$ and $\Gap$, respectively. Furthermore, by \cite[Ch. II, Theorem 3]{Se}, we know that
\begin{equation} \label{amalgam}
\Gamma = \Ga\ast_{\Gap}\hat{\Gamma}^D_0(M),
\end{equation}
that is, $\Gamma$ is the amalgamated product of the stabilizers of $v_\ast$ and $\hat{v}_\ast$ over the stabilizer of $e_\ast$.

The free abelian group $\Z[\cl E^+]$ over $\E^+$ can be canonically identified, via projection, with the quotient $C_\cl E$ of $\Z[\cl E]$ by the relations $e+\bar e=0$ for all $e\in\cl E$. Setting $C_\cl V:=\Z[\cl V]$, we obtain a short exact sequence
\begin{equation} \label{exseq}
0\longrightarrow C_\cl E\overset{\varphi}\longrightarrow C_\cl V\xrightarrow{\deg}\Z\longrightarrow0
\end{equation}
where $\varphi(e):=t(e)-s(e)$ and $\deg$ is the degree map.

If $X$ and $A$ are sets write $\cF(X, A)$ for the set of functions from $X$ to $A$. Now suppose that $A$ is an abelian group; there are two degeneracy maps
\[ \begin{array}{rccc}
   \varphi_s,\varphi_t:&\cF(\mathcal E,A) & \lra & \cF(\mathcal V,A)\\[3mm]
   &\nu&\longmapsto&\Big(\varphi_s(\nu):v\mapsto\sum_{s(e)=v}\nu(e)\Big)\\[3mm]
   &\nu&\longmapsto&\Big(\varphi_t(\nu):v\mapsto\sum_{t(e)=v}\nu(e)\Big).
   \end{array} \]
Put
\[ \cF_0(\mathcal E,A):=\bigl\{\nu\in\cF(\mathcal E,A)\mid\text{$\nu(\bar e)=-\nu(e)$ for all $e\in \mathcal E$}\bigr\}. \]
An $A$-valued \emph{harmonic cocycle} is a function $\nu\in\cF_0(\mathcal E,A)$ such that $\varphi_s(\nu)=0$; we write $\HC(A)$ for the abelian group of $A$-valued harmonic cocycles.

Finally, assume further that $A$ is a left $G$-module for some subgroup $G$ of $\PGL_2(\Q_p)$. Then $\cF(\E, A)$ and its submodules $\cF_0(\E, A)$ and $\HC(A)$ are endowed with a structure of left $G$-modules by the rule ${}^{g}\nu(e):=g\cdot\nu(g^{-1}e)$. The next result is proved in \cite[\S 8]{Gr}.

\begin{lemma}[Greenberg] \label{GreenbergsLemma}
The sequence of $\G $-modules
\[ 0\longrightarrow\HC(A)\longrightarrow\cF_0(\mathcal E,A)\overset{\varphi_s}{\longrightarrow}\cF(\mathcal V, A)\longrightarrow0 \]
is exact.
\end{lemma}

By applying Shapiro's lemma, the short exact sequence of Lemma \ref{GreenbergsLemma} induces a long exact sequence
\begin{equation} \label{Grlong}
\begin{split}
0\longrightarrow\HC(A)^{\G }&\longrightarrow A^{\Gamma_0^D(pM)}\longrightarrow(A\times A)^{\G_0(M)}\\
&\longrightarrow H^1\bigl(\G,\HC(A)\bigr)\overset{\varrho}{\longrightarrow}H^1\bigl(\G,\cF_0(\E,A)\bigr),
\end{split}
\end{equation}
where
\[ \mathrm{Im}(\varrho)\simeq H^1\bigl(\G _0(pM),A\bigr)_{\text{$p$-{\rm new}}}:=\ker\Big(H^1\bigl(\G _0(p M),A\bigr)\lra H^1\bigl(\G _0(M),A\bigr)^2\Big). \]
The group $\GL_2(\Q_p)$ acts on the left on $\PP^1(\Q_p)$ by fractional linear transformations and this action, as before, factors through $\PGL_2(\Q_p)$.

Set $U_{e_\ast}:=\Z_p$. Since $\GL_2(\Q_p)$ acts transitively on $\E$ and the stabilizer of $e_\ast$ in $\GL_2(\Q_p)$ is $\GL_2(\Z_p)$, we may define a map from $\cl E$ to the family of compact open subsets of $\PP^1(\Q_p)$ by
\[ e\longmapsto U_e:=\gamma(U_{e_\ast}), \]
where $\g\in\GL_2(\Q_p)$ is any element such that $e=\gamma(e_\ast)$. Notice that $U_{\bar e}=\PP^1(\Q_p)-U_e$. The sets $\{U_e\}_{e\in\cl E}$ form a basis of compact open subsets for the $p$-adic topology of $\PP^1(\Q_p)$.

Let $A$ be a free module of finite rank over either $\Z$ or $\Z_p$, equipped with a left action of a subgroup $G$ of $\PGL_2(\Q_p)$. Let $\cM (A):=\Meas \bigl(\PP^1(\Q_p),A\bigr)$ denote the space of $A$-valued measures on $\PP^1(\Q_p)$ and write $\cM_0(A)\subset \cM(A)$ for the submodule of measures of total mass $0$. Define a left action of $\Gamma$ on $\cM(A)$ by imposing that
\[ (\gamma\cdot\nu)(U):=\nu\bigl(\gamma^{-1}(U)\bigr) \]
for all compact open subsets $U$ of $\PP^1(\Q_p)$. Thanks to the above observation (see also, e.g., \cite[\S 2.3]{Das} and \cite[Lemma 27]{Gr}), there is a canonical isomorphism of $G$-modules
\begin{equation} \label{*}
\HC(A)\overset{\simeq}{\lra}\cM_0(A),\qquad c\longmapsto\nu_c
\end{equation}
given by the rule $\nu_c(U_e):=c(e)$.

\subsection{Construction of the measure-valued 1-cocycle} \label{sec-def-meas}

From here until the end of the paper we assume that $D>1$. In this subsection we define a measure-valued cohomology class $\boldsymbol\mu$ which will be a crucial ingredient for our purposes. The construction of $\boldsymbol{\mu}$ will be done in stages.

Choose a system $\mathcal Y$ of representatives for the cosets $\Gap\backslash\G$. Since $\G$ acts transitively on $\E^+$ and $\Gap$ is the stabilizer of $e_\ast$, we have $\mathcal Y=\{\g_e\}_{e\in\E^+}$ with $\g_e\in\G$ such that $\g_e(e)=e_\ast$. Any other system of representatives is of the form $\mathcal Y'=\{\g'_e\}_{e\in\E^+}$ with
\begin{equation} \label{f(e)-eq}
\g'_e=f(e)\g_e
\end{equation}
for a suitable $f(e)\in\Gap$.

\begin{definition} \label{muniv}
The \emph{universal $1$-cochain associated with $\mathcal Y$} is the $1$-cochain
\[ \mu^{\mathcal Y}_{\rm univ}:\G\longrightarrow\cF_0\big(\mathcal E,H_1\big(\Gap,\Z\big)_T\big)\simeq\cF_0\big(\mathcal E,\Gap^{\rm ab}_T\big) \]
determined for all $\gamma\in\Gamma$ by the following rules:
\begin{itemize}
\item for all $e\in\cl E^+$ let $g_{\g,e}\in\Gap$ be defined by the equation $\g_e\g=g_{\gamma,e}\g_{\g^{-1}(e)}$, then set
\[ \mu^{\mathcal Y}_{\rm univ}(\g)(e):=\bigl[g_{\gamma,e}\bigr]; \]
\item for all $e\in\cl E^-$ set
\[ \mu^{\mathcal Y}_{\rm univ}(\g)(e):=-\mu^{\mathcal Y}_{\rm univ}(\g)(\bar e). \]
\end{itemize}
\end{definition}
As in the introduction, let
\[ H=H_0^D:=\bigl[H_1\bigl(\Gamma_0^D(pM),\Z\bigr)_T\big/{\rm Im}(\pi^\ast)\bigr]_T. \]
Fix a non-zero torsion-free quotient $\mathbb H$ of $H$ and let
\[ \pi_{\mathbb H}:H_1\bigl(\Gap,\Z\bigr)\simeq\Gap^{\rm ab}\longrightarrow\mathbb H \]
be the quotient map. In subsequent sections we will specialize to $\mathbb H=H$, which represents the most relevant case for this article. However, in connection with \cite[Conjecture 2]{Gr}, other interesting instances arise for $\mathbb H=H_1(A,\Z)$ where $A_{/\Q}$ is a modular abelian variety (e.g., an elliptic curve) that is a $p$-new quotient of $J_0^D(pM)$.

Let $\mu^{\mathcal Y}_{\mathbb H}\in C^1\bigl(\G,\cF_0(\mathcal E,\mathbb H)\bigr)$ be the $1$-cochain defined, in terms of the universal $1$-cochain of Definition \ref{muniv}, by
\begin{equation} \label{mH}
\mu^{\mathcal Y}_{\mathbb H}(\g)(e):=\pi_{\mathbb H}\bigl(\mu^{\mathcal Y}_{\rm univ}(\g)(e)\bigr)
\end{equation}
for all $\g\in\Gamma$ and all $e\in\E$.
%The minus sign introduced in \eqref{mH} is convenient for some of the calculations we will perform later, in particular those of \S \ref{sec-ord}, where we compute the $p$-adic valuation of a certain integration map.
The following properties of $\mu^{\mathcal Y}_{\mathbb H}$, whose verification is easy but somewhat tedious, will be used repeatedly.

\begin{proposition} \label{3}
\begin{enumerate}
\item[(i)] The cochain $\mu^{\mathcal Y}_{\mathbb H}$ lies in $Z^1\bigl(\G,\cF_0(\mathcal E,\mathbb H)\bigr)$, i.e., it is a $1$-cocycle.
\item[(ii)] The class of $\mu^{\mathcal Y}_{\mathbb H}$ in $H^1\bigl(\G,\cF_0(\mathcal E,\mathbb H)\bigr)$ is independent of the choice of $\mathcal Y$.
\item[(iii)] If $\nu\in Z^1\bigl(\G,\cF_0(\mathcal E,\mathbb H)\bigr)$ is cohomologous to $\mu^{\mathcal Y}_{\mathbb H}$ then there exists a system of representatives $\mathcal Y'$ for $\Gap\backslash\Gamma$ such that $\nu=\mu^{\mathcal Y'}_{\mathbb H}$.
\end{enumerate}
\end{proposition}

We will denote the class of $\mu^{\mathcal Y}_{\mathbb H}$ in $H^1\bigl(\G,\cF_0(\mathcal E,\mathbb H)\bigr)$ by $\boldsymbol{\mu}^{\mathcal Y}_{\mathbb H}$; although, by part (ii) of the proposition above, this class is independent of the choice of a system of representatives, we keep the superscript $\mathcal Y$ in the notation because we reserve the unadorned symbol for a slightly different cohomology class (cf. Definition \ref{def-mu}).

\begin{proof} Part (i) follows straightly by unwinding the definition of $\mu^{\mathcal Y}_{\mathbb H}$. As for (ii), a direct computation reveals that if $\mathcal Y'$ is another system of representatives for $\Gap\backslash\Gamma$ then
\[ \mu^{\mathcal Y}_{\mathbb H}-\mu^{\mathcal Y'}_{\mathbb H}=\delta([f]), \]
the coboundary associated with the function $[f]:\E\rightarrow\mathbb H$ such that $[f](e):=\pi_{\mathbb H}\bigl([f(e)]\bigr)$ for $e\in\E^+$ and $[f](e):=-\pi_{\mathbb H}\bigl([f(\bar e)]\bigr)$ for $e\in\E^-$; here $f(e)$ is as in \eqref{f(e)-eq}. Finally, to prove claim (iii) let $g$ be a function in $\cF(\E^+,\mathbb H)=\cF_0(\E,\mathbb H)$ whose image under the cobounday map is $\mu^{\mathcal Y}_{\mathbb H}-\nu$, and let
\[ f':\E^+\longrightarrow\Gap \]
be an arbitrary lift of $g$; then it can be checked that $\nu=\mu^{\mathcal Y'}_{\mathbb H}$ for $\mathcal Y':=\bigl\{f'(e)\g_e\bigr\}_{e\in\E^+}$. \end{proof}

Now recall the map
\[ \varrho:H^1\bigl(\G,\HC(\mathbb H)\bigr)\longrightarrow H^1\bigl(\G,\cF_0(\E,\mathbb H)\bigr) \]
from \eqref{Grlong}, with $A=\mathbb H$.

\begin{lemma} \label{lemma-univ-I}
The class $\boldsymbol{\mu}^{\mathcal Y}_{\mathbb H}$ lies in $\mathrm{Im}(\varrho)$.
\end{lemma}

\begin{proof} By \eqref{Grlong} and Shapiro's lemma, there are exact sequences fitting in the commutative diagram
\[ \footnotesize{\xymatrix{\dots\ar[r]& H^1\bigl(\G,\HC(\mathbb H)\bigr)\ar[r]^-{\varrho}\ar[d]^-\simeq & H^1\bigl(\G,\cF_0(\E,\mathbb H)\bigr)\ar[r]\ar[d]^-\simeq&H^1\bigl(\G,\cF(\mathcal V,\mathbb H)\bigr)\ar[r]\ar[d]^-\simeq&\dots\\
                 \dots\ar[r]& H^1\bigl(\Gap,\mathbb H\bigr)_{p{\text{-new}}}\ar[r]& H^1\bigl(\Gap,\mathbb H\bigr)\ar[r]& H^1\bigl(\Ga,\mathbb H\bigr)\times H^1\bigl(\hat{\Gamma}_0^D(M),\mathbb H\bigr)\ar[r] & \dots}} \]
Let $\mathcal Y$ be any system of representatives for $\Gap\backslash\Gamma$. The class in $H^1\bigl(\Gap,\mathbb H\bigr)$ corresponding to $\boldsymbol{\mu}^{\mathcal Y}_{\mathbb H}$ under the above isomorphism can be represented by the cochain
\[ g\in\Gap\longmapsto\mu^{\mathcal Y}_{\mathbb H}(g)(e_\ast)\in\mathbb H \]
which, according to Definition \ref{muniv}, is equal to $\pi_{\mathbb H}\bigl([g]\bigr)$. If $G\in\bigl\{\Gamma_0^D(M),\Gamma_0^D(pM),\hat\Gamma_0(M)\bigr\}$ then $G$ acts trivially on $\mathbb H$, so there is a canonical isomorphism
\[ H^1(G,\mathbb H)\simeq\Hom\bigl(H_1(G,\Z),\mathbb H\bigr). \]
Under these identifications, the map in the lower right corner of the above diagram is
\[ \begin{array}{ccc}
   H^1\bigl(\Gap,\mathbb H\bigr) & \longrightarrow & H^1\bigl(\Ga,\mathbb H\bigr)\times H^1\bigl(\hat{\Gamma}_0^D(M),\mathbb H\bigr)\\[2mm]
   f & \longmapsto & f\circ\Big(\mathrm{cor}_{\Gap}^{\Ga},\mathrm{cor}_{\Gap}^{\hat{\Gamma}_0^D(M)}\Big)
   \end{array} \]
with $\mathrm{cor}$ indicating corestriction. Now observe that for $\mathbb H=H$ there is an equality of maps
\[ \Big(\mathrm{cor}_{\Gap}^{\Ga},\mathrm{cor}_{\Gap}^{\hat{\Gamma}_0^D(M)}\Big)=(\pi_1^*,\pi_2^*) \]
where the $\pi_i^\ast$ for $i=1,2$ are the pull-backs defined in the introduction. Since $\mathbb H$ is a quotient of
\[ H:=\bigl[H_1\bigl(\Gap,\Z\bigr)_T\big/\mathrm{Im}(\pi_1^*)+\mathrm{Im}(\pi_2^*)\bigr]_T, \]
we deduce that the image of $\boldsymbol{\mu}^{\mathcal Y}_{\mathbb H}$ in $H^1\bigl(\Ga,\mathbb H\bigr)\times H^1\bigl(\hat{\Gamma}_0^D(M),\mathbb H\bigr)$ is trivial, and the lemma is proved. \end{proof}

\begin{remarkwr} \label{cohomologous-remark}
Some words of caution are in order here: Lemma \ref{lemma-univ-I} does \emph{not} prove that the cocycle $\mu^{\mathcal Y}_{\mathbb H}$ lies in $Z^1\bigl(\G,\HC(\mathbb H)\bigr)$. Rather, it only shows that some cocycle cohomologous to it takes values in $\HC(\mathbb H)$. However, by part (iii) of Proposition \ref{3} this implies that there do exist choices of $\mathcal Y$ such that $\mu^{\mathcal Y}_{\mathbb H}$ belongs to $Z^1\bigl(\G,\HC(\mathbb H)\bigr)$.
\end{remarkwr}

The last observation in Remark \ref{cohomologous-remark} motivates the following

\begin{definition} \label{harmf}
A system of representatives $\mathcal Y$ for $\Gap\backslash\Gamma$ is said to be \emph{harmonic} if $\mu^{\mathcal Y}_{\mathbb H}$ belongs to $Z^1\bigl(\G,\HC(\mathbb H)\bigr)$.
\end{definition}

Let us introduce a class of systems of representatives for the cosets $\Gap\backslash\Gamma$ which can be explicitly constructed and shown to be harmonic. This construction will be useful in \S \ref{SHecke} but may be also of independent interest, as it is amenable to explicit calculations: building on the computational tools developed in \cite{GV}, our recipe can be implemented in order to compute the lattice of $p$-adic periods that we introduce in Section \ref{periods}.

\begin{definition} \label{radial}
A system of representatives $\mathcal Y=\{\g_e\}_{e\in\E^+}$ for $\Gap\backslash\Gamma$ is called \emph{radial} if the two conditions
\begin{enumerate}
\item $\{\g_e\}_{s(e)=v}=\{\g_i\g_v\}_{i=0}^p$ for all $v\in\V^+$,
\item $\{\g_e\}_{t(e)=v}=\{\tilde{\g}_i\g_v\}_{i=0}^p$ for all $v\in\V^-$
\end{enumerate}
hold for suitable choices of sets of representatives $\{\g_i\}_{i=0}^p$, $\{\tilde{\g}_i\}_{i=0}^p$, $\{\g_v\}_{v\in\V^+}$ and $\{\g_v\}_{v\in\V^-}$ for the cosets $\Gap\backslash\Ga$, $\Gap\backslash\hat{\Gamma}_0^D(M)$, $\Ga\backslash\G$ and $\hat{\Gamma}_0^D(M)\backslash\G$, respectively, such that $\g_0=\tilde{\g}_0=\g_{v_\ast}=\g_{\hat{v}_\ast}=1$.
\end{definition}

The next result justifies the formal introduction of the notion of radial systems.

\begin{proposition} \label{radial-lemma}
Radial systems of representatives exist and are harmonic.
\end{proposition}

\begin{proof} The existence of radial systems follows from the fact that $\mathcal T$ is a tree. More precisely, for any choice of sets of representatives $\{\g_i\}_{i=0}^p$ and $\{\tilde{\g}_i\}_{i=0}^p$ of $\Gap\backslash\Ga$ and $\Gap\backslash\hat{\Gamma}_0^D(M)$, respectively, with $\g_0=\tilde{\g}_0=1$ conditions $(1)$ and $(2)$ in Definition \ref{radial} uniquely determine sets $\{\g_v\}_{v\in\V^+}$ and $\{\g_v\}_{v\in\V^-}$ satisfying them.

Let us now prove that radial systems are harmonic. According to Lemma \ref{GreenbergsLemma}, we need to show that, with slightly abusive but self-explaining notation, $\nu:=\varphi_s\bigl(\mu^{\mathcal Y}_{\mathbb H}\bigr)\in Z^1\bigl(\G,\cF(\V,\mathbb H)\bigr)$ is identically zero. Firstly, notice that
\begin{equation} \label{u}
\nu_{\g}(v_\ast)=0\quad\text{for all $\g\in\Ga$},\qquad\nu_{\hat\g}(\hat{v}_\ast)=0\quad\text{for all $\hat\g\in\hat{\Gamma}_0^D(M)$}.
\end{equation}
Indeed, once again with a slight abuse of notation, for $\g\in\Ga$ one has
\[ \nu_\g(v_\ast)=\sum_{s(e)=v_\ast}\bigl[g_{\g,e}\bigr]=\Big[\mathrm{cor}_{\Ga}^{\Gap}([\g])\Big]=\pi_1^*([\g])\in\pi_1^*\bigl(H_1\bigl(X^D_0(M),\Z\bigr)\bigr), \]
hence the image of $\nu_{\g}(v_\ast)$ in $\mathbb H$ vanishes. Similar considerations apply to elements $\hat\g$ in $\hat{\Gamma}_0^D(M)$.

Secondly, one has
\begin{equation} \label{dos}
\nu_{\g_v}(v_\ast)=0\qquad\text{for all $v\in\V^+$},\qquad\nu_{\g_v}(\hat{v}_\ast)=0\quad\text{for all $v\in\V^-$}.
\end{equation}
In fact, with notation as in Definition \ref{radial}, for $v\in\V^+$ there are equalities
\[ \nu_{\g_v}(v_\ast)=\sum_{i=0}^p\mu^{\mathcal Y}_{\mathbb H,\g_v}(\g_i^{-1}e_\ast)=\sum_i{}^{\g_i}\mu^{\mathcal Y}_{\mathbb H,\g_v}(e_\ast)=\sum_i\bigl(\mu^{\mathcal Y}_{\mathbb H,\g_i \g_v}-\mu^{\mathcal Y}_{\mathbb H,\g_i}\bigr)(e_\ast), \]
and this vanishes in $\mathbb H$ because the $\g_i$ and the $\g_i\g_v$ belong to $\mathcal Y$ by definition of $\mu_{\mathbb H}^{\mathcal Y}$.

Similarly, if $v\in\V^-$ then
\[ \nu_{\g_v}(\hat{v}_\ast)=\sum_{i=0}^p\mu^{\mathcal Y}_{\mathbb H,\g_v}(\tilde{\g}_i^{-1}\bar{e}_\ast)=\sum_i{}^{\tilde{\g}_i}\mu^{\mathcal Y}_{\mathbb H,\g_v}(\bar{e}_\ast)= \sum_i\bigl(\mu^{\mathcal Y}_{\mathbb H,\tilde{\g}_i\g_v}-\mu^{\mathcal Y}_{\mathbb H,\tilde{\g}_i}\bigr)(\bar{e}_\ast), \]
which is again trivial because the $\tilde{\g}_i$ and the $\tilde{\g}_i\g_v$ are in $\mathcal Y$.

This is enough to prove the lemma, as one can check that $\nu$ is uniquely determined by conditions (\ref{u}) and (\ref{dos}). \end{proof}

Let $\mathcal Y$ be an arbitrary harmonic system. Before proceeding with our arguments, we make an observation which will prove useful later.

\begin{remarkwr} \label{ker-varrho-remark}
The analogue of part (ii) of Proposition \ref{3} for $\mu_{\mathbb H}^{\mathcal Y}$ does \emph{not} hold true in $H^1\bigl(\G,\HC(\mathbb H)\bigr)$. Indeed, there exist several choices of harmonic systems $\mathcal Y'$ such that the classes of $\mu_{\mathbb H}^{\mathcal Y}$ and $\mu_{\mathbb H}^{\mathcal Y'}$ in $H^1\bigl(\G,\HC(\mathbb H)\bigr)$ are different; this is due to the fact that $\ker(\varrho)$ is not trivial. More precisely, it is immediate to check from \eqref{Grlong} that
\[ \ker(\varrho)=(\mathbb H\times\mathbb H)/\mathbb H_0 \]
where $\mathbb H_0$ is the image of $\mathbb H$ in $\mathbb H\times\mathbb H$ under the embedding $a\mapsto\bigl((p+1)a,-(p+1)a\bigr)$.
\end{remarkwr}

Fix once and for all, for the rest of this article, a prime $r\nmid pDM$ and set
\[ t_r:=T_r-r-1, \]
which we regard as an operator in either $\mathcal H(M)$, $\mathcal H(p M)$ or $\mathcal H(p,M)$ according to the context.

\begin{definition} \label{def-mu}
The class $\boldsymbol{\mu}_{\mathbb H}$ is the image of $t_r\cdot\mu^{\mathcal Y}_{\mathbb H}$ in $H^1\bigl(\G,\HC(\mathbb H)\bigr)$.
\end{definition}

Dropping $\mathcal Y$ from the notation is justified by the following

\begin{lemma} \label{Eis}
The class $\boldsymbol{\mu}_{\mathbb H}$ is independent of the choice of $\mathcal Y$.
\end{lemma}

\begin{proof} It suffices to show that $t_r$ vanishes on the kernel of $\varrho$, i.e., that $\ker(\varrho)$ is an Eisenstein submodule of $H^1\bigl(\G,\HC(\mathbb H)\bigr)$. If this is true then $t_r$ also acts on $\mathrm{Im}(\varrho)\hookrightarrow H^1\bigl(\G,\cF_0(\E,\mathbb H)\bigr)$, and the lemma follows from part (ii) of Proposition \ref{3}.

As pointed out in Remark \ref{ker-varrho-remark}, $\ker(\varrho)$ is equal to the image of $\mathbb H\times\mathbb H=H^0\bigl(\G,\cF(\V,\mathbb H)\bigr)$ in $H^1\bigl(\G,\HC(\mathbb H)\bigr)$. Since Hecke operators commute with the connecting maps of the long exact sequence \eqref{Grlong} by \cite[Lemma 1.1.1]{AS}, it is enough to show that $H^0\bigl(\G,\cF(\V,\mathbb H)\bigr)$ is Eisenstein.

Let $f\in H^0\bigl(\G,\cF(\V,\mathbb H)\bigr)$. According to Section \ref{section-Hecke}, $T_r(f)=\sum_{i=0}^{r+1}s_i\cdot f$ where the $s_i\in R(pM)$ are elements of norm $r$. Since the elements in $R(pM)$ fix both $v_\ast$ and $\hat{v}_\ast$, it follows that
\[ T_r(f)(v_\ast)=(r+1)f(v_\ast),\qquad T_r(f)(\hat{v}_\ast)=(r+1)f(\hat{v}_\ast). \]
Since $T_r(f)$ is again $\G$-invariant, it is completely determined by these two values. Hence $T_r(f)=(r+1)f$, and we are done. \end{proof}

In light of isomorphism \eqref{*}, we shall denote by $\boldsymbol{\mu}_{\mathbb H}$ also the measure-valued cohomology class in $H^1\bigl(\G,\cM_0(\mathbb H)\bigr)$ associated with $\boldsymbol{\mu}_{\mathbb H}$. In the special case where $\mathbb H=H$, we denote $\boldsymbol{\mu}_H$ simply by $\boldsymbol{\mu}$.

\section{Multiplicative integration pairings} \label{raav}

\subsection{An integration pairing for Shimura curves} \label{ShimuraPairing}

As in \S \ref{sec-def-meas}, let $\mathbb H$ be a non-zero torsion-free quotient of $H$, which now we further assume to be stable for the action of $\mathcal H(pM)$. This holds for all the cases we are interested in, like $\mathbb H=H$ or $\mathbb H=H_1(A,\Z)$ where $A_{/\Q}$ is a modular abelian variety that is a $p$-new quotient of $J_0^D(pM)$.

The aim of this section is to introduce a suitable analogue of the integration pairing defined by Dasgupta in \cite[\S 3.2]{Das}. Notice though that when $D>1$ there is no natural action of $\Gamma$ on $\Div\PP^1(\Q)$ and consequently Dasgupta's pairing makes no sense. Instead, following ideas of
Greenberg (\cite{Gr}), we shall construct a pairing
\[ \langle\,,\rangle:H_1(\Gamma,\cD)\times H^1\bigl(\Gamma,\cM_0(\mathbb H)\bigr)\longrightarrow\C_p^\times\otimes\mathbb H \]
where, for notational convenience, from here on we set
\[ \cD:=\Div^0\cl H_p. \]
Notice that if $\mathbb H=H$ then $\C_p^\times\otimes\mathbb H=T(\C_p)$. Let $\cC\bigl(\PP^1(\Q_p),\C_p\bigr)$ denote the $\C_p$-algebra of $\C_p$-valued continuous functions on $\PP^1(\Q_p)$; since it is naturally a submodule of $\cF\bigl(\PP^1(\Q_p),\C_p\bigr)$, it inherits a left action of $\GL_2(\Q_p)$. The multiplicative group $\cC\bigl(\PP^1(\Q_p),\C_p\bigr)^\times$ of invertible elements of $\cC\bigl(\PP^1(\Q_p),\C_p\bigr)$ consists of the $\C_p^\times$-valued functions in $\cC\bigl(\PP^1(\Q_p),\C_p\bigr)$. As in \cite[Definition 2.2]{Das}, given a function $f\in\cC\bigl(\PP^1(\Q_p),\C_p\bigr)^\times$ and a measure $\nu\in \cM_0(\mathbb H)$ we define the \emph{multiplicative integral} of $f$ against $\nu $ as a limit of Riemann products
\[ \mint_{\PP^1(\Q_p)}fd\nu:=\lim_{\|\mathcal U\|\rightarrow0}\prod_{U\in\mathcal U}f(t_U)\otimes\nu(U)\in\C_p^\times\otimes\mathbb H. \]
In the above formula the limit is taken over finer and finer covers $\mathcal U$ of $\PP^1(\Q_p)$ by compact open disjoint subsets, and $t_U$ is an arbitrary point of $U$ for every $U\in \mathcal U$. The limit converges in $\C_p^\times\otimes\mathbb H$ because $\nu$ is a measure. This produces a pairing
\begin{equation} \label{pairing-1}
(\,,):\cC\bigl(\PP^1(\Q_p),\C_p\bigr)^\times\times \cM_0(\mathbb H)\longrightarrow\C_p^\times\otimes\mathbb H.
\end{equation}
One can easily verify that the pairing \eqref{pairing-1} satisfies
\[ (\gamma\cdot f,\gamma\cdot\nu)=(f,\nu) \]
for all $\g\in\GL_2(\Q_p)$, $f\in \cC\bigl(\PP^1(\Q_p),\C_p\bigr)^\times$ and $\nu\in\cM_0(\mathbb H)$. Since the multiplicative integral of a non-zero constant against a measure $\nu\in\cM_0(\mathbb H)$ is $1$, the above pairing induces another pairing
\begin{equation}\label{pairing-2}
(\,,):\cC\bigl(\PP^1(\Q_p),\C_p\bigr)^\times\big/\C_p^\times\times\cM_0(\mathbb H)\longrightarrow\C_p^\times\otimes\mathbb H.
\end{equation}
For any $d\in\cD$ let $f_d$ denote a rational function on $\PP^1(\C_p)$ such that $\mathrm{div}(f_d)=d$. The function $f_d$ is not unique; more precisely, it is well defined only modulo multiplication by constant non-zero functions. Since the divisor $d$ is not supported on $\PP^1(\Q_p)$, the function $f_d$ restricts to a function in $\cC\bigl(\PP^1(\Q_p),\C_p\bigr)^\times$, which will be denoted in the same fashion by an abuse of notation. Thus the map $d\mapsto f_d$ defines an embedding
\[ \cD\;\longmono\;\cC\bigl(\PP^1(\Q_p),\C_p\bigr)^\times\big/\C_p^\times \]
which is invariant under the natural left actions of $\GL_2(\Q_p)$. Hence, composing this injection with \eqref{pairing-2} yields a $\GL_2(\Q_p)$-invariant pairing (denoted, by a slight abuse of notation, by the same symbol)
\begin{equation} \label{pairing-3}
(\,,):\cD\times\cM_0(\mathbb H)\longrightarrow\C_p^\times\otimes\mathbb H,\qquad(d,\mu):=\mint_{\PP^1(\Q_p)}f_d\,d\mu
\end{equation}
which, by construction, factors naturally through $(\cD\otimes\cM_0(\mathbb H))_\Gamma$. By cap product, we finally obtain the desired pairing
\begin{equation} \label{pairing}
\langle\,,\rangle:H_1(\Gamma,\cD)\times H^1\bigl(\Gamma,\cM_0(\mathbb H)\bigr)\longrightarrow\C_p^\times\otimes\mathbb H.
\end{equation}

\subsection{Hecke-equivariance of the integration map}\label{SHecke}

Recall from above that $\cD:=\Div^0\mathcal H_p$ and let also $\boldsymbol{\mu}_{\mathbb H}$ be as in \S \ref{sec-def-meas}. Fixing $\boldsymbol{\mu}_{\mathbb H}$ in the second variable of the pairing $\langle\,,\rangle$ of \eqref{pairing} yields a homomorphism
\begin{equation} \label{integral}
\int:H_1(\Gamma,\cD)\longrightarrow \C_p^{\times }\otimes \mathbb H.
\end{equation}
The group $H_1(\Gamma,\cD)$ is an $\mathcal H(p,M)$-module, while $\C_p^{\times }\otimes \mathbb H$ is naturally an $\mathcal H(pM)$-module, because of our assumptions on $\mathbb H$. Our present aim is to prove the following

\begin{proposition} \label{HeckeEquiv}
The integration map $\int$ is equivariant for the actions of the Atkin--Lehner involutions $W_p$ and $W_\infty$ and of the Hecke operators $T_\ell$ with $\ell\nmid pDM$.
\end{proposition}

We devote the rest of this subsection to the proof of this proposition. Let
\[ T\in\bigl\{T_\ell\mid\ell\nmid pDM\bigr\}\cup\{W_p, W_\infty\} \]
and let $\mathcal Y$ be a harmonic system of representatives for $\Gap\backslash\Gamma$; we want to show that
\[ \big\langle T\cdot c,t_r\cdot\mu_{\mathbb H}^{\mathcal Y}\big\rangle=T\cdot\big\langle c,t_r\cdot\mu_{\mathbb H}^{\mathcal Y}\big\rangle \]
for all $c\in H_1(\G,\cD)$. Thanks to \eqref{Heckepairing} and the commutativity of the Hecke algebras, this is equivalent to showing that
\begin{equation} \label{hec}
\big\langle T\cdot c,\mu_{\mathbb H}^{\mathcal Y}\big\rangle=T\cdot\big\langle c,\mu_{\mathbb H}^{\mathcal Y}\big\rangle
\end{equation}
for all $c\in t_r\cdot H_1(\G,\cD)$. Note that, by Lemma \ref{Eis} and \eqref{Heckepairing} again, it follows that both $\big\langle T\cdot c,\mu_{\mathbb H}^{\mathcal Y}\big\rangle$ and $\big\langle c,\mu_{\mathbb H}^{\mathcal Y}\big\rangle$ are independent of the chosen harmonic system $\mathcal Y$.

Let $W\in\{W_p,W_\infty\}$ denote any of the two involutions. We shall prove \eqref{hec} by computing the two sides of the equality by means of two different choices of harmonic systems $\mathcal Y$.

In both Hecke algebras $\mathcal H(p,M)$ and $\mathcal H(pM)$ one has that $W=T(\om)$ for an element $\om\in R(pM)$ satisfying $\G\om=\om\G$ and $\Gap\om=\om\Gap$. On $H_1(\Gamma,\cD)$ the involution $W$ acts as
\[ c=\sum_kd_k[\g_k]\longmapsto\sum_k(\om^{-1}d_k)[\om^{-1}\g_k\om],\quad\text{$d_k\in\cD$ for all $k$}, \]
hence

\begin{equation} \label{w1}
\big\langle W\cdot c,\mu_{\mathbb H}^{\mathcal Y}\big\rangle=\mint f_{\om^{-1} d_k}(t)d\mu_{\mathbb H,\om^{-1}\g_k\om}^{\mathcal Y}(t)=\lim_{\mathcal U}\prod_k\prod_{U\in\mathcal U}f_{d_k}(t_U)\otimes\mu_{\mathbb H,\om^{-1}\g_k\om}^{\mathcal Y}(\om^{-1}U).
\end{equation}
On the other hand, $W$ acts on $\mathbb H$ simply by conjugation by $\om$, so that

\begin{equation} \label{w2}
W\cdot\big\langle c,\mu_{\mathbb H}^{\mathcal Y}\big\rangle=\lim_{\mathcal U}\prod_k\prod_{U\in\mathcal U}f_{d_k}(t_U)\otimes\om^{-1}\mu^{\mathcal Y}_{\mathbb H,\g_k}(U)\om.
\end{equation}
Given a radial (hence harmonic, by Lemma \ref{radial-lemma}) system $\mathcal Y=\{\g_e\}_{e\in\E^+}$, let us introduce the system
\[ {}^\om\mathcal Y:=\bigl\{\om\g_{\om^{-1}(e)}\om^{-1}\bigr\}_{e\in\E^+}. \]
Notice that ${}^\om\mathcal Y$ is again radial, because conjugation by $\om_\infty$ (respectively, $\om_p$) leaves each of $\Gap$, $\Ga$, $\hGa$ and $\G$ invariant (respectively, leaves $\Gap$ and $\G$ invariant and interchanges $\Ga$ and $\hGa$). Again by Lemma \ref{radial-lemma} we obtain that ${}^\om\mathcal Y$ is harmonic, and thus the above observations apply.

If one computes \eqref{w1} with respect to $\mathcal Y$ and computes \eqref{w2} with respect to ${}^\om\mathcal Y$ it follows that \eqref{w1} is equal to \eqref{w2}, as we wished to show.

Now let $\ell\nmid pDM$ be a prime number and fix a radial system $\mathcal Y$.

\begin{lemma} \label{new}
Let $\mu^{(\ell)}_{\mathbb H}\in Z^1\bigl(\G,\cF_0(\E,\mathbb H)\bigr)$ be the cocycle determined by the rule
\[ \mu^{(\ell)}_{\mathbb H,\g}(e):=\sum_i\pi_\mathbb H\big(\bigl[t_i(g_{\g,e})\bigr]\big) \]
for every $\g\in\G$ and every even edge $e\in\E^+$. Then
\begin{enumerate}
\item[(i)] $\mu^{(\ell)}_{\mathbb H}$ is a cocycle which takes values in $\HC(\mathbb H)$;
\item[(ii)] $T_\ell\bigl(\mu_{\mathbb H}^{\mathcal Y}\bigr)=\mu^{(\ell)}_{\mathbb H}+b$ for some $b\in\ker(\varrho)\subset Z^1\bigl(\G,\HC(\mathbb H)\bigr)$.
\end{enumerate}
\end{lemma}

\begin{proof} For simplicity, write $\nu:=\mu_{\mathbb H}^{\mathcal Y}\in Z^1\bigl(\G,\HC(\mathbb H)\bigr)$ and $\nu^{(\ell)}:=\mu^{(\ell)}_{\mathbb H}$. Set $I(\ell):=\{0,\dots,\ell\}$. An easy computation shows that
\[ T_\ell(\nu)_\g=-\sum_i\g g_j\cdot\nu_{t_i^{-1}(\g)} \]
for all $\g\in\G$, where $j=j(i)$ is the permutation of $I(\ell)$ such that $t_i(\g)=g_i^{-1}\g g_j$.

For every edge $e\in\mathcal E^+$ one has
\[ \bigl(\g g_j\cdot\nu_{t_i^{-1}(\g)}\bigr)(e)=\pi_\mathbb H\big(\bigl[g_{t_i^{-1}(\g),g_j^{-1}\g^{-1}e}\bigr]\big) \]
with $g_{t_i^{-1}(\g),g_j^{-1}\g^{-1}e}\in\Gamma_0^D(pM)$ satisfying the equation
\begin{equation} \label{eq++}
\g_{g_j^{-1}\g^{-1}e}t_i^{-1}(\g)=g_{t_i^{-1}(\g),g_j^{-1}\g^{-1}e}\g_{t_ig_j^{-1}\g^{-1}(e)}=g_{t_i^{-1}(\g),g_j^{-1}\g^{-1}e}\g_{g_i^{-1}(e)}.
\end{equation}
For every $g\in\GL_2(\Q_p)$ and every $\g\in\G$ with $g^{-1}\g g\in\G$ there exists $h_{g,e}\in\G_0^D(pM)$ such that $\g_{g^{-1}(e)}=h_{g,e}g^{-1}\g_e g$. Using the equality $\g_{\g^{-1}(e)}\g^{-1}=g_{\g^{-1},\g^{-1}(e)}\g_e$, one shows that
\[ \g_{g_j^{-1}\g^{-1}e}t_i^{-1}(\g)=h_{g_j,\g^{-1}(e)}\bigl(g_{j}^{-1}g_{\g^{-1},\g^{-1}(e)}g_i\bigr)h^{-1}_{g_i,e}\g_{g_i^{-1}(e)}. \]
Comparing with formula \eqref{eq++}, we deduce that
\[ g_{t_i^{-1}(\g),g_j^{-1}\g^{-1}e}=h_{g_j,\g^{-1}(e)}\bigl(g_j^{-1}g_{\g^{-1},\g^{-1}(e)}g_i\bigr)h^{-1}_{g_i,e}. \]
Since $g_{t_i^{-1}(\g),g_j^{-1}\g^{-1}e}$, $h_{g_j,\g^{-1}(e)}$ and
$h^{-1}_{g_i,e}$ are in $\G_0^D(pM)$, we conclude that
$g_j^{-1}g_{\g^{-1},\g^{-1}(e)}g_i$ belongs to $\G_0^D(pM)$ as well.
Accordingly, in $H$ we have
\[ \bigl[g_{t_i^{-1}(\g),g_j^{-1}\g^{-1}e}\bigr]=\bigl[g_j^{-1}g_{\g^{-1},\g^{-1}(e)}g_i\bigr]+\bigl[h_{g_j,\g^{-1}(e)}\bigr]-[h_{g_i,e}]. \]
An easy calculation now yields that $g_{\g,e}^{-1}=g_{\g^{-1}(e),\g^{-1}}$. Hence
\[ T_\ell(\nu)_\g(e)=\sum_i\pi_\mathbb H\bigl([t_i(g_{\g,e})]\bigr)-\sum_i\pi_\mathbb H\bigl(\bigl[h_{g_i,\g^{-1}(e)}\bigr]\bigr)+\sum_i\pi_\mathbb H\bigl([h_{g_i,e}]\bigr). \]
Let us introduce the function
\[ \varphi:\mathcal E^+\longrightarrow \mathbb H,\qquad e\longmapsto\sum_i\pi_\mathbb H\bigl([h_{g_i,e}]\bigr) \]
and extend it to an element of $\mathcal F_0(\mathcal E,\mathbb H)$ by the obvious recipe. Since
\[ (\gamma\varphi)(e)=\varphi\bigl(\gamma^{-1}(e)\bigr)=\sum_i\pi_\mathbb H\bigl(\bigl[h_{g_i,\g^{-1}(e)}\bigr]\bigr), \]
it follows that the cocycle $\nu^{(\ell )}$ represents the same class as $T_\ell(\nu)$ in $Z^1\bigl(\G,\mathcal F_0(\mathcal E,\mathbb H)\bigr)$. In other words, the class of $b:=T_\ell(\nu)-\nu^{(\ell )}$ in $H^1\bigl(\G,\mathcal F_0(\mathcal E,\mathbb H)\bigr)$ is trivial.

Let us now prove that $\nu^{(\ell)}\in Z^1\bigl(\G,\HC(\mathbb H)\bigr)$. In order to show this, write $i\mapsto\sigma(i)$ for the permutation of $I(\ell)$ such that
$t_i(g_{\g,e})=g_ig_{\g,e}g_{\sigma(i)}$. Note that
\[ \sum_i[t_i(g_{\g,e})]=\sum_i\bigl[g_i^{-1}g_{\g,e}g_{\sigma(i)}\bigr]=\sum_{s\in S}\bigl[g_s^{-1}g_{\g,e}^{m_s}g_s\bigr]=\sum_{s\in S}m_s\bigl[g_s^{-1}g_{\g,e}g_s\bigr] \]
where $S$ is a suitable subset of $I(\ell)$ and $m_s\in\Z$ for all $s\in S$. Therefore it suffices to show that the cocycle defined on $\Gamma$ by the rule
\[ \g\longmapsto\Big(e\mapsto\pi_\mathbb H\big(\bigl[g_s^{-1}g_{\g,e}g_s\bigr]\big)\Big) \]
for $e\in\E^+$ is harmonic. Keep the notation of Definition \ref{radial} for the radial system $\mathcal Y$. For every $s\in S$ define
\[ H_s:=g_s^{-1}\G_0^D(pM)g_s,\qquad\G_s:=g_s^{-1}\G g_s \]
as subgroups of $\GL_2(\Q_p)$. Then a system of representatives for the cosets $H_s\backslash\G_s$ is given by the set
\[ \bigl\{\g'_e:=g_s^{-1}\g_i\g_vg_s\bigr\}. \]
Arguing as in Lemma \ref{radial-lemma}, one immediately shows that the cocycle in $Z^1\bigl(g_s^{-1}\G g_s,\mathcal F_0(\mathcal E,\mathbb H)\bigr)$ defined on $e\in\E^+$ by the rule
\[ g_s^{-1}\g g_s\longmapsto\Big(e\mapsto\pi_\mathbb H\bigl(\bigl[g'_{\g,e}\bigr]\bigr)\Big), \]
where $\g'_eg_s^{-1}\g g_s=g'_{\g,e}\g'_{e'}$, is harmonic. Since $g'_{\g,e}=g_s^{-1}g_{\g,e}g_s$, this is enough to conclude that $\nu^{(\ell )}$ takes values in $\mathcal F_{\rm har}(\mathbb H)$ as well. Hence $b$ actually lies in $Z^1\big(\G,\mathcal F_{\rm har}(\mathbb H)\big)$. By the above observation, if $\boldsymbol{b}$ is the class of $b$ in $H^1\bigl(\Gamma,\HC(\mathbb H)\bigr)$ then $\varrho(\boldsymbol{b})=0$, as we wanted. \end{proof}

Now the equivariance of the integration map under $T_\ell$ follows easily. In fact, keeping the notation introduced before, Lemma \ref{new} implies that
\[ \langle T_\ell(c),\nu\rangle=\langle c,T_\ell(\nu)\rangle=\langle c,\nu^{(\ell)}\rangle=T_\ell\bigl(\langle c,\nu\rangle\bigr), \]
which concludes the proof of Proposition \ref{HeckeEquiv}.

\subsection{The $p$-adic valuation of the integration map} \label{sec-ord}

Unless otherwise stated, for the rest of the article set $\mathbb H:=H$. Let
\[ {\rm red}:\mathcal H_p\longrightarrow\cl T \]
be the $\GL_2(\Q_p)$-equivariant reduction map which is described, e.g., in \cite[I.2]{BC} and choose a base point $\tau\in K_p-\Q_p$ such that ${\rm red}(\tau)=v_\ast$.

Let $\g_1,\g_2$ be two arbitrary elements of $\G$. Let $\{e_1,\dots,e_{n}\}$ be the {\em even geodesic} joining $v_\ast$ with ${\rm red}\big(\gamma_1(\tau)\big)=\gamma_1(v_\ast)\in \V^+$. By this we mean that $e_i\in\cl E^+$ are even edges such that
\begin{itemize}
\item $s(e_1)=v_\ast$, $s(e_{n})=\gamma_1(v_\ast)=:v_n$;
\item $t(e_i)=t(e_{i+1})=:v_i$ for \emph{odd} indices in $\{1,\dots,n-1\}$;
\item $s(e_i)=s(e_{i+1})=:v_i$ for \emph{even} indices in $\{2,\dots,n-2\}$.
\end{itemize}
Notice that, above, the integer $n$ is always even. It is our aim here to prove the following result, which will be used in the next section.

\begin{proposition} \label{eq-II}
Keep notation as above. If the $g_i$ for $i=1,\dots,n$ are elements of $\Gap$ such that $\gamma_{e_i}\gamma_2=g_i\gamma_{e'_i}$ with $e'_i\in\E^+$ then
\[ \ord_p\bigg(\mint_{\PP^1(\Q_p)}\frac{t-\gamma_1(\tau)}{t-\tau}d\mu_{H,\gamma_2}^{\mathcal Y}(t)\bigg)=\sum_{i=1}^{n}(-1)^i[g_i]\in H. \]
\end{proposition}

In order to prove the formula in the proposition, let $\tau_0:=\tau$, $\tau_n:=\gamma_1(\tau)$ and for every $i=1,\dots,n-1$ choose $\tau_i\in K_p-\Q_p$ such that ${\rm red}(\tau_i)=v_i$. Since
\[ \frac{t-\gamma_1(\tau)}{t-\tau}=\frac{t-\tau_n}{t-\tau_{n-1}}\cdot\frac{t-\tau_{n-1}}{t-\tau_{n-2}}\cdot\dots\cdot\frac{t-\tau_1}{t-\tau_0}, \]
it is easy to check that
\[ \ord_p\bigg(\mint\frac{t-\gamma_1(\tau)}{t-\tau}d\mu_{H,\gamma_2}^{\mathcal Y}(t)\bigg)=\sum_{i=0}^{n-1}\mint\ord_p\left(\frac{t-\tau_{i+1}}{t-\tau_i}\right)d\mu_{H,\gamma_2}^{\mathcal Y}(t). \]
Proposition \ref{eq-II} now follows recursively from the next computation.

\begin{lemma} \label{lemma-ord-I}
Let $v_1,v_2\in\V$ be consecutive vertices and let $\tau_1,\tau_2\in K_p-\Q_p$ be such that ${\rm red}(\tau_i)=v_i$ for $i=1,2$. Set $e:=(v_1,v_2)$ if $v_1\in\V^+$ and $e:=(v_2,v_1)$ otherwise. If $\gamma\in\G$ then
\[ \mint\ord_p\left(\frac{t-\tau_{2}}{t-\tau_1}\right)d\mu_{H,\gamma}^{\mathcal Y}(t)=\begin{cases}-[g] & \text{if $v_1\in \V^+$}\\[2mm][g] & \text{if $v_1\in \V^-$}\end{cases} \]
where $g\in \Gap$ is such that $\gamma_e \gamma=g\gamma_{e'}$ for some $e'\in\cl E^+$.
\end{lemma}

\begin{proof} Let us give the details only for $v_1\in \V^+$, the other case being analogous. Consider the points $\tau_{\hat{v}_\ast}:=\g_e(\tau_2)$ and $\tau_{v_\ast}:=\g_e(\tau_1)$, so that
\[ {\rm red}(\tau_{v_\ast})=v_\ast,\qquad {\rm red}(\tau_{\hat{v}_\ast})=\hat{v}_\ast. \]
Thanks to the $\G$-equivariance of \eqref{pairing-1}, we have
\[ \mint\ord_p\left(\frac{t-\tau_2}{t-\tau_1}\right)d\mu_{H,\gamma}^{\mathcal Y}(t)=\mint\ord_p\left(\frac{t-\tau_{\hat{v}_\ast}}{t-\tau_{v_\ast}}\right)
\cdot d({}^{\g_e}\mu_{H,\gamma}^{\mathcal Y})(t). \]
Now, by \cite[I.2]{BC} (see also \cite[p. 444]{Das}), there is an equality
\[ \ord_p\left(\frac{t-\tau_{\hat{v}_\ast}}{t-\tau_{v_\ast}}\right)=\begin{cases}-1 & \text{if $t\in\Z_p$}\\[2mm]0 & \text{if $t\not\in\Z_p$}\end{cases}, \]
hence
\[ \mint\ord_p\left(\frac{t-\tau_2}{t-\tau_1}\right)d\mu_{H,\gamma}^{\mathcal Y}(t)=-\mu_{H,\gamma}^{\mathcal Y}(\g_e^{-1}\Z_p). \]
By definition, we have
\[ \mu_{H, \gamma}^{\mathcal Y}\bigl(\g_e^{-1}\cdot\Z_p\bigr)=\mu_{H,\gamma}^{\mathcal Y}\bigl(\g_e^{-1}(e_\ast)\bigr)=[g] \]
where $\g_e\gamma=g\gamma_{\g^{-1}(e)}$. Thus we find that
\[ \mint\ord_p\left(\frac{t-\tau_2}{t-\tau_1}\right)d\mu_{H,\gamma}^{\mathcal Y}(t)=-[g], \]
which is the searched-for equality. \end{proof}

\section{The lattice of $p$-adic periods} \label{periods}

From the long exact sequence in $\Gamma$-homology associated with the short exact sequence
\[ 0\longrightarrow\cD\longrightarrow\Div\mathcal H_p\xrightarrow{\text{deg}}\Z\longrightarrow0 \]
we extract the boundary homomorphism
\[ \partial:H_2(\Gamma,\Z)\longrightarrow H_1(\Gamma,\cD). \]
Set
\[ \Phi:=\int\circ\;\partial:H_2(\Gamma,\Z)\longrightarrow T(\C_p) \]
and let $L$ be the image of $\Phi$ in $T(\C_p)$.

\begin{proposition} \label{prop-lattice}
The module $L$ is contained in $T(\Q_p)$ and is preserved by the action of the Hecke algebra.
\end{proposition}

\begin{proof} Let $F$ be a non-trivial finite extension of $\Q_p$. Any point $\tau\in F-\Q_p$ can be used as base point in order to compute the map $\partial$ on $Z_2(\G,\Z)$; explicitly, one has
\[ \partial\left(\sum_ia_i[\gamma_{i,1}|\gamma_{i,2}]\right)=\sum_i\bigl(\gamma_{i,1}^{-1}(\tau)-\tau\bigr)\otimes a_i[\gamma_{i,2}] \]
on a generic $2$-cycle in $Z_2(\G,\Z)$. This shows that $\partial\bigl(H_2(\Gamma,\Z)\bigr)\subset H_1\bigl(\Gamma,\Div^0\mathcal H_p(F)\bigr)$, and from the very definition of the integration pairing it then follows that $L\subset T(F)$. Since this holds for all finite extensions $F$ of $\Q_p$, we deduce that $L$ is contained in $T(\Q_p)$.

Finally, the submodule $L$ is invariant under the action of the Hecke operators because the map $\Phi$ is Hecke equivariant. In fact, the boundary map $\partial$ is Hecke equivariant by, e.g., \cite[Lemma 5.1.3]{DasPhD} (one just needs to formally replace Dasgupta's $\Delta_\Q=\PGL_2(\Q)$ with $\Gamma$), and the integration map $\int$ is Hecke equivariant as well by Proposition \ref{HeckeEquiv}. \end{proof}

Set
\[ \pi_*:=(\pi_1)_*\oplus(\pi_2)_*:H_1\bigl(X_0^D(pM),\Z\bigr)\longrightarrow H_1\bigl(X_0^D(M),\Z\bigr)^2, \]
where $(\pi_i)_*$ is the push-forward of the map $\pi_i$ for $i=1,2$.

\begin{lemma} \label{lemma-incl}
There is a canonical injection $\lambda:\ker(\pi_*)\hookrightarrow H$ which has finite cokernel and is equivariant for the action of $W_\infty$.
\end{lemma}

\begin{proof} As endomorphisms of $H_1\bigl(X_0^D(M),\Z\bigr)^2$, there is an equality
\[ \pi_*\circ\pi^*=\begin{pmatrix}p+1&T_p\\T_p&p+1\end{pmatrix}. \]
But the eigenvalues of $T_p$ are bounded by $2\sqrt p$, so the above endomorphism is injective and has finite cokernel. A formal argument concludes the proof. \end{proof}

There is yet another way to interpret $\ker(\pi_*)$. Namely, applying Shapiro's lemma to the long exact sequence in homology attached to \eqref{exseq} gives an exact sequence of abelian groups
\begin{equation} \label{ex-seq-iso}
H_2\bigl(\Gamma_0^D(M),\Z\bigr)^2\longrightarrow H_2(\G,\Z)\overset{\theta}{\longrightarrow}H_1\bigl(\Gamma_0^D(pM),\Z\bigr)\overset{\pi_*}{\longrightarrow}H_1\bigl(\Gamma_0^D(M),\Z\bigr)^2,
\end{equation}
with $\theta$ being the connecting homomorphism; it follows that
\begin{equation} \label{lemma-image}
\ker(\pi_*)=\mathrm{Im}(\theta).
\end{equation}
Below, fix $\tau\in K_p-\Q_p$ such that $\mathrm{red}(\tau)=v_\ast$ and use it to compute the map $\partial$ on $Z_2(\G,\Z)$ as in the proof of Proposition \ref{prop-lattice}.

\begin{proposition} \label{prop-int-diag}
The diagram
\[ \xymatrix@C=35pt@R=30pt{H_2(\G,\Z)\ar[r]^-\theta\ar[d]^-\partial\ar[dr]^-\Phi & \ker(\pi_*)\ar[r]^-\lambda & H\ar[d]^-{-t_r}\\
                           H_1(\G,\cD)\ar[r]^-\int & T(K_p)\ar[r]^-{\ord_p} & H} \]
is commutative.
\end{proposition}

\begin{proof} Thanks to relation \eqref{Heckepairing} and the obvious commutativity
\begin{equation} \label{commu}
\ord_p\circ t_r=t_r\circ\ord_p:T(\Q_p)=\Q_p^\times\otimes H\longrightarrow H,
\end{equation}
it suffices to show that the diagram
\[ \xymatrix@C=35pt@R=30pt{H_2(\G,\Z)\ar[r]^-\theta\ar[d]^-\partial & \ker(\pi_*)\ar[r]^-\lambda & H\ar^-{-\mathrm{id}}[d]\\
                           H_1(\G,\cD)\ar[r]^-\int & T(K_p)\ar[r]^-{\ord_p} & H} \]
is commutative. To compute the integration map $\int$, fix a harmonic (e.g., radial) system of representatives $\mathcal Y$ for $\Gap\backslash\Gamma$.

Let $a=\sum_ia_i\bigl[\gamma_{i,1}|\gamma_{i,2}\bigr]$, with $a_i\in\Z$, be an element of $Z_2(\G,\Z)$; it follows from Proposition \ref{eq-II} and the definitions of $\partial$ and of pairing \eqref{pairing} that
\begin{equation} \label{eq-comm-II}
\begin{split}
\ord_p\Big(\int\partial(a)\Big)=\ord_p\bigl(\langle\partial(a),\boldsymbol{\mu}\rangle\bigr)&=\sum_ia_i\ord_p\bigg(\mint\frac{t-\gamma_{i,1}^{-1}(\tau)}{t-\tau}d\mu_{H,\gamma_{i,2}}^{\mathcal Y}(t)\bigg)\\ &=\sum_i\sum_{j=1}^{n_i}(-1)^j a_i[g_{i,j}]
\end{split}
\end{equation}
where the $g_{i,j}\in\Gap$ are defined as follows:
\begin{itemize}
\item consider an even geodesic $\{e_1^{(i)},\dots,e_{n_i}^{(i)}\}$ from $v_\ast$ to $\gamma_{i,1}^{-1}(v_\ast)$;
\item define $g_{i,j}\in\Gap$ via the equation $\gamma_{e_j^{(i)}}\cdot\gamma_{i,2}=g_{i,j}\cdot\gamma_{e'_{i,\,j}}$ for some $e'_{i,\,j}\in\cl E^+$.
\end{itemize}
This accounts for half of the above diagram. As for the other half, the map $\lambda$ is just the restriction to $\ker(\pi_*)$ of the projection $H_1\bigl(\Gap,\Z\bigr)\rightarrow H$, whereas the explicit description of $\theta$ is somewhat more involved, since $\theta$ is the composition of the connecting homomorphism $H_2(\G,\Z)\rightarrow H_1(\G,C_\cl E)$ in the long exact sequence attached to \eqref{exseq} with the isomorphism $H_1(\G,C_\cl E)\simeq H_1\bigl(\Gamma_0^D(pM),\Z\bigr)$ provided by Shapiro's lemma. By unwinding definitions and writing down explicit expressions of these two maps at the level of chains, one obtains that
\begin{equation} \label{eq-com-I}
\lambda\bigl(\theta(a)\bigr)=\sum_i\sum_{e\in\E^+}\alpha_{e,i}\cdot a_i\otimes[g_{e,i}]
\end{equation}
where
\begin{itemize}
\item $E_i=\sum_{e\in\E^+}\alpha_{e,i}\cdot e\in C_\E$, with $\alpha_{e,i}\in\Z$, is such that $\varphi(E_i)=\gamma_{i,1}^{-1}(v_\ast)-v_\ast$;
\item $\gamma_e\cdot\gamma_{i,2}=g_{e,i}\cdot\gamma_{e'_i}$ for $e'_i\in\cl E^+$ and $g_{e,i}\in\Gamma_0^D(M)$.
\end{itemize}
Here recall from \eqref{exseq} that $\varphi(e):=t(e)-s(e)$. In our case, for all $i$ we may choose
\[ E_i:=\sum_{j=1}^{n}(-1)^{j-1}e_j^{(i)}\in\Z[\cl E^+]. \]
The claim of the proposition follows immediately by comparing \eqref{eq-comm-II} and \eqref{eq-com-I}. \end{proof}

We can now prove the main result of this section.

\begin{theorem} \label{lattice}
The submodule $L$ of $T(\Q_p)$ is a lattice of rank $2g$.
\end{theorem}

\begin{proof} According to \cite[\S 4.2]{Pa}, it suffices to show that the image of $L$ under the map
\[ \ord_p:T(\Q_p)\longrightarrow H\subset H\otimes\R \]
is a lattice of rank $2g$ in the $\R$-vector space $H\otimes\R$. Since $H_2(\G,\Z)$ is a finitely generated abelian group, the same is true of $\ord_p(L)$. Moreover, by construction, $H$ is a free discrete submodule of $H\otimes\R$, hence $\ord_p(L)$ is a free discrete submodule of $H\otimes\R$ as well. Now, extending our previous notation by linearity, observe that
\[ \mathrm{rank} _\Z\,\bigl(\ord_p(L)\bigr)=\dim_{\Q_p}\left(\ord_p\circ\int\circ\;\partial\bigl(H_2(\G,\Q_p)\bigr)\right). \]
By Proposition \ref{prop-int-diag}, we know that
\[ \ord_p\circ\int\circ\;\partial=-(t_r\circ\lambda\circ\theta)\otimes_\Z\Q_p. \]
The map $\lambda\otimes\Q_p$ is surjective by Lemma \ref{lemma-incl}, while so is $t_r\otimes\Q_p$ because the absolute values of the Hecke operator $T_r$ acting on $H$ are bounded from above by $2\sqrt{r}$. Combining this with \eqref{lemma-image}, we obtain that the image of $\ord\circ\int\circ\;\partial$ is $H\otimes\Q_p$, whose dimension over $\Q_p$ is $2g$. \end{proof}

\section{The $p$-adic uniformization} \label{uniformization-section}

\subsection{The main theorem}

The ultimate goal of this section is to prove Theorem \ref{GreenbergConj2}, which represents the main contribution of this article. We start by observing that, in analogy with \cite{Das}, Theorem \ref{GreenbergConj2} is a consequence of the following result.

\begin{theorem} \label{prop-I}
The equality of maps $\cl L_p^D\cdot\ord_p=\log_p$ holds on the lattice $L$.
\end{theorem}

For the convenience of the reader, let us explain why Theorem \ref{prop-I} implies Theorem \ref{GreenbergConj2}; we follow \cite[pp. 449--450]{Das} closely. For any $\Z[W_\infty]$-module $M$ and sign $\epsilon\in\{\pm1\}$ we set $M_\epsilon:=M/(W_\infty-\epsilon)$. This applies in particular to the Hecke module $H$; define
\[ T_\epsilon:=\mathbb G_m\otimes_\Z H_\epsilon.\]
Since, as shown in the proof of Proposition \ref{prop-lattice}, the map $\Phi$ is equivariant for the action of $W_\infty$ and the cokernel of the canonical map $H\rightarrow H_+\oplus H_-$ is supported at $2$, it follows that there exists an isogeny of $2$-power degree
\[ T/L\longrightarrow T_+/L_+\oplus T_-/L_- \]
of rigid analytic tori over $\Q_p$.

Fix a sign $\epsilon$; we prove Theorem \ref{GreenbergConj2} by showing that $T_\epsilon/L_\epsilon$ admits a Hecke-equivariant isogeny over $K_p$ to the rigid analytic space associated with $J_0^D(pM)^{\text{$p$-{\rm new}}}$.

Notice that it follows from Lemma \ref{lemma-incl} that $\ker(\pi_*)\otimes\Q$ is canonically isomorphic to $H\otimes\Q$ and that there is a canonical injection
$\lambda_\epsilon:\ker(\pi_*)_\epsilon\hookrightarrow H_\epsilon$. Write
\[ \T:=\T_0^D(M)\subset\End(H\otimes\Q) \]
for the image in $\End(H\otimes\Q)$ of the Hecke algebra $\mathcal H(pM)$.

Now we freely use the notation of \S \ref{singular-subsec}; in particular, $X$ is the group of degree zero divisors on the set of supersingular points of $X_0^D(M)$ in characteristic $p$ and $X^*$ is its $\Z$-dual. Since $X\otimes\Q$, $X^*\otimes\Q$, $H_\epsilon\otimes\Q$ and $\ker(\pi_*)_\epsilon\otimes\Q$ are free $\T$-algebras of rank one (see \cite[Ch. 1]{helm}), we can choose Hecke-equivariant maps $\xi_\epsilon$ and $\eta_\epsilon$ making the diagram
\begin{equation} \label{dia-I}
\xymatrix@C=30pt@R=30pt{\ker(\pi_*)\ar[d]^-{\eta_\epsilon}\ar[r]^-{\lambda_\epsilon} & H_\epsilon\ar[r]^-{-t_r} & H_\epsilon\ar[d]^-{\xi_\epsilon} \\
X\ar[rr]^-{\ord_X} & & X^\ast}
\end{equation}
commute.

The map $\Phi$, being $W_\infty$-equivariant, restricts to a map $\Phi_\epsilon:H_2(\G,\Z)_\epsilon\rightarrow T_\epsilon(K_p)$. Consider the diagram
\begin{equation} \label{dia-II}
\xymatrix@C=30pt@R=30pt{H_2(\G,\Z)_\epsilon\ar[r]^-{\Phi_\epsilon}\ar[d]^-{\theta_\epsilon} & T_\epsilon(K_p)\ar[r]^-{\xi_\epsilon} & X^*\otimes K_p^\times\\
                        \ker(\pi_*)_\epsilon\ar[rr]^-{\eta_\epsilon} & & X\ar[u]^{j}}
\end{equation}
with the map $\theta$ having already made its appearance in the exact sequence \eqref{ex-seq-iso}.

In the statement below, let
\[ \ord_p,\log_p:X^*\otimes K_p^\times\longrightarrow X^*\otimes\Z_p \]
denote the usual valuation and logarithm maps.

\begin{proposition} \label{prop-comm}
Diagram \eqref{dia-II} commutes up to elements in $\ker(\log_p)\cap\ker(\ord_p)$.
\end{proposition}

\begin{proof} To begin with, the map $\ord_p\circ\Phi_\epsilon$ is equal to $-t_r\circ\lambda_\epsilon\circ\theta_\epsilon$ by Proposition \ref{prop-int-diag}. Thus
\[ \ord_p\circ\xi_\epsilon\circ\Phi_\epsilon=-\xi_\epsilon\circ t_r\circ\lambda_\epsilon\circ\theta_\epsilon=\ord_X\circ\eta_\epsilon\circ\theta_\epsilon=\ord_p\circ
j\circ\eta_\epsilon\circ\theta_\epsilon, \]
where the first equality follows from the commutativity of $\ord_p$ and $\xi_\epsilon$, the second is a consequence of the commutativity of diagram \eqref{dia-I} and the third follows from the definition of $\ord_X$. Hence diagram \eqref{dia-II} commutes up to elements in $\ker(\ord_p)$. Since the maps in diagram \eqref{dia-II} are Hecke equivariant, Theorem \ref{prop-I} and Proposition \ref{prop-II} imply that diagram \eqref{dia-II} is commutative also up to elements in $\ker(\log_p)$, from which the result follows. \end{proof}

Notice that $\ker(\log_p)\cap\ker(\ord_p)$ is a finite subgroup of $X^*\otimes K_p^\times$ whose order is supported at the prime divisors of $p-1$ if $p>2$ (at $2$ if $p=2$). We are ready to prove Theorem \ref{GreenbergConj2}, which we reformulate below in terms of the $\epsilon$-components.

\begin{theorem} \label{main-thm-epsilon}
There is a Hecke-equivariant isogeny over $K_p$ between $T_\epsilon/\pi_\epsilon(L)$ and the rigid analytic space associated with $J_0^D(pM)^{\text{$p$-{\rm new}}}$ whose degree is divisible only by $2$ and the primes dividing the order of $\ker(\log_p)\cap\ker(\ord_p)$ and ${\rm coker}(\eta_\epsilon\otimes{\rm id})$ for $\epsilon=\pm$.
\end{theorem}

\begin{proof} Recall exact sequence \eqref{unif-I}, which gives a rigid-analytic uniformization of $J_0^D(M)^{\text{$p$-new}}$ in terms of $X$ and the map $j$. Since $\ker(\log_p)\cap\ker(\ord_p)$ is finite, Proposition \ref{prop-comm} shows that the map $\xi_\epsilon$ induces an isogeny
\[ T_\epsilon/\pi_\epsilon(L)\longrightarrow J_0^D(M)^{\text{$p$-new}} \]
which is defined over $K_p$. The Hecke-equivariance is immediate and the statement on the degree of the isogeny follows from the bounds given above. \end{proof}

As already mentioned, Theorem \ref{main-thm-epsilon} immediately implies Theorem \ref{GreenbergConj2} in the introduction. Furthermore, thanks to the Hecke-equivariance of the isogeny in the theorem above, a proof of Greenberg's conjecture \cite[Conjecture 2]{Gr} is also a consequence of Theorem \ref{main-thm-epsilon}: see \S \ref{greenberg-subsec} for details.

\bigskip

The rest of the article will be devoted to proving Theorem \ref{prop-I}.

\subsection{A lifting theorem for measure-valued cohomology classes} \label{sec-meas}

As a piece of notation, in the sequel write $\cM:=\cM(H\otimes \Z_p)$ for the $\Z_p$-module of measures on $\PP^1(\Q_p)$ (of arbitrary total mass) with values in $H\otimes\Z_p$.

Define $\Y:=\Z_p^2$; we view the elements of $\Y$ as column vectors $\bigl(\begin{smallmatrix}x\\y\end{smallmatrix}\bigr)$ -- sometimes written as rows only for notational convenience -- and let the semigroup $\M_2(\Z_p)$ act on $\Y$ by left multiplication, so that
\[ \gamma\cdot\xi:=(ax+by,cx+dy) \]
for every $\gamma=\smallmat abcd\in\M_2(\Z_p)$ and every $\xi=\bigl(\begin{smallmatrix}x\\y\end{smallmatrix}\bigr)\in\Y$.

Similarly as before, write $\cM_{\Y}$ for the $\Z_p$-module of measures on $\Y$ with values in $H\otimes \Z_p$. The above action can be used to define a left action of $\M_2(\Z_p)$ on $\cM_{\Y}$, like the one introduced in \S \ref{BT}. If $\nu\in\cM_{\Y}$ we let $\Supp(\nu)$ denote the support of $\nu$, and we say that $\nu$ is supported on a compact open subset $\U$ of $\Y$ if $\Supp(\nu)\subset\U$. For any compact open subset $\U$ of $\Y$ we denote by $\cM_\U$ the $\Z_p$-submodule of $\cM_{\Y}$ consisting of those measures supported on $\U$. It is immediate to check that if $\gamma\in M_2(\Z_p)$ and $\nu\in\cM_\U$ then $\gamma\cdot\nu\in\cM_{\gamma\U}$.

Let $\X:=(\Z_p^2)'$ denote the set of primitive vectors in $\Y$, that is, the set of elements $(a,b)\in\Y$ such that $a$ and $b$ are not both divisible by $p$. Again, write $\cM_{\X}$ for the $\Z_p$-module of $H\otimes \Z_p$-valued measures on $\X$. We omit the proof of the following

\begin{lemma} \label{stable-kernel-lemma}
The kernel of the canonical projection $q:\cM_{\Y}\rightarrow\cM_{\X}$ is preserved by the action of $\M_2(\Z_p)$.
\end{lemma}

As a consequence of Lemma \ref{stable-kernel-lemma}, one can define a left action of $\M_2(\Z_p)$ on $\cM_\X$ by the formula
\[ \gamma\cdot\nu:=q\big(\gamma\cdot i(\nu)\big)=q(\gamma\cdot\tilde\nu) \]
for every $\tilde\nu$ such that $q(\tilde\nu)=\nu$.

Now consider the map
\[ \pi:\X\longrightarrow\PP^1(\Q_p),\qquad (a,b)\longmapsto [(a,b)], \]
whose fibers are principal homogeneous spaces for $\Z_p^\times$. The action of $\M_2(\Z_p)$ on $\Y$ restricts to an action of $\GL_2(\Z_p)$ on $\X$ and $\pi$ is a homomorphism of left $\GL_2(\Z_p)$-modules, where the left action of $\GL_2(\Z_p)$ on $\PP^1(\Q_p)$ is by fractional linear transformations. The fibration $\pi$ induces by push-forward a map $\pi_*:\cM_\X\rightarrow\cM$ where $\pi_*(\nu):=\nu\bigl(\pi^{-1}(U)\bigr)$ for every $\nu\in\cM_\X$ and every compact open subset $U$ of $\PP^1(\Q_p)$. With a slight abuse of notation, we then get a map $\pi_*$ from $Z^1\bigl(\Gamma_0^D(M),\cM_\X\bigr)$ to $Z^1\bigl(\Gamma_0^D(M),\cM\bigr)$ by the rule $\pi_*(\nu)_\g:=\pi_*(\nu_\g)$, and finally a map
\[\pi_*:H^1\bigl(\Gamma_0^D(M),\cM_\X\bigr)\longrightarrow H^1\bigl(\Gamma_0^D(M),\cM\bigr). \]
Recall the class $\boldsymbol{\mu}\in H^1(\Gamma,\mathcal M_0(H))$ defined in \S \ref{sec-def-meas}, which we can naturally regard now as an element of $H^1(\Gamma,\mathcal M)$.

\begin{theorem} \label{prop-lifting}
There exists $\boldsymbol{\tilde{\mu}}\in H^1\bigl(\Gamma_0^D(M),\cM_\X\bigr)$ such that $\pi_*(\boldsymbol{\tilde\mu})$ is the restriction of $\boldsymbol{\mu}$ to $\Gamma_0^D(M)$.
\end{theorem}

In order to prove Theorem \ref{prop-lifting}, for every integer $r\geq1$ let $\X_r$ be the set of primitive vectors in $(\Z_p/p^r\Z_p)^2$, again endowed with the natural left action of $\GL_2(\Z_p)$. One immediately verifies that $\X\simeq\invlim\X_r$ with respect to the canonical projection maps.

For $r\geq1$ let $\Gamma_r:=\Gamma^D_1(p^r)\cap\Ga$, which is a congruence subgroup of $\Gamma^D_0(p^rM)$.

\begin{proposition} \label{isos}
For every $r\geq1$ set
\[ \tilde{U}_r:=\bigl\{(x,y)\in\X\mid x\in1+p^r\Z_p, y\in p^r\Z_p\bigr\}\subset\X \]
and
\[ U_r:=\bigl\{[x:y]\in\PP^1(\Q_p)\mid x\in1+p^r\Z_p, y\in p^r\Z_p\bigr\}\subset\PP^1(\Q_p). \]
The maps
\begin{itemize}
\item[(i)] $H^1\bigl(\Gamma_0^D(M),\cM_\X\bigr)\overset{\simeq}{\longrightarrow}\invlim_rH^1(\Gamma_r,H\otimes \Z_p),\qquad\tilde{\mu}\mapsto\{\tilde{\mu}_r\}_{r\geq1},\quad\tilde{\mu}_r(\g):=\tilde{\mu}_\g(\tilde{U}_r)$,
\item[(ii)] $H^1\bigl(\Gamma_0^D(M),\cM\bigr)\overset{\simeq}{\longrightarrow}\invlim_rH^1\bigl(\Gamma^D_0(p^rM),H\otimes \Z_p\bigr),\qquad\mu\mapsto\{\mu_r\}_{r\geq1},\quad\mu_r(\g ):= \mu_\g(U_r)$
\end{itemize}
are isomorphisms.
\end{proposition}

\begin{proof} We provide details for (i) only, as (ii) is completely analogous. For all $r\geq1$ the action of $\GL_2(\Z_p)$ on $\X_r$ is transitive and the stabilizer of $(1,0)$ is the subgroup $\Sigma(p^r)$ consisting of the matrices $\smallmat abcd$ with $a\equiv1\pmod{p^r}$ and $c\equiv0\pmod{p^r}$. Thus the map $\smallmat abcd \mapsto(a,c)$ describes a bijection between the set of classes $\GL_2(\Z_p)/\Sigma(p^r)$ and $\X_r$.

Let $\cM_{\X_r}$ be the $\Z_p$-module of $H\otimes \Z_p$-valued measures on $\X_r$. Since $\X_r$ is a finite set, the module $\cM_{\X_r}$ identifies canonically with the $\Z_p$-module of $H\otimes \Z_p$-valued functions on $\X_r$. Then we have a canonical isomorphism of $\GL_2(\Z_p)$-modules
\begin{equation} \label{eq-D}
\cM_{\X}\simeq\invlim_r\cM_{\X_r},\qquad\mu\longmapsto\bigl[v\mapsto\mu\bigl(v+p^r\Z_p^2\bigr)\bigr]
\end{equation}
where the inverse limit is computed with respect to the norm maps $\nu_r:\cM_{\X_r}\rightarrow\cM_{\X_{r-1}}$ which, for $r\geq2$, are defined by $\mu\mapsto[x\mapsto \sum_{\pi_r(y)=x}\mu(y)]$. Here $\pi_r:\X_r\rightarrow \X_{r-1}$ stands for the canonical projection.

Note that $\Gamma_0^D(M)$ injects into $\GL_2(\Z_p)$ via $\iota_p$, and in this way it acts on $\X_r$. Since $\Gamma_0^D(M)$ is dense in $\GL_2(\Z_p)$ with respect to the $p$-adic topology, it follows that the action of $\Gamma_0^D(M)$ on $\X_r$ induced by $\iota_p$ is transitive. Hence, since $\Gamma_r=\Sigma(p^r)\cap\Gamma_0^D(M)$, for all $r\geq1$ there exists, as above, a bijection between $\Gamma_0^D(M)/\Gamma_r$ and $\X_r$.

The set of functions $\cM_{\X_r}$ is then identified with the set of functions from the cosets $\Gamma_0^D(M)/\Gamma_r$ to $H\otimes \Z_p$, which is in bijection with the set of functions $\phi:\Gamma_0^D(M)\rightarrow H\otimes \Z_p$ such that $\phi(\gamma\cdot\delta)=\phi(\gamma)$ for all $\gamma\in\Gamma_0^D(M)$ and all $\delta\in\Gamma_r$, namely, the coinduced $\Gamma_0^D(M)$-module ${\rm Coind}_{\Gamma_r}^{\Gamma_0^D(M)}(H\otimes \Z_p)$. Thus Shapiro's lemma shows that
\begin{equation} \label{equation-shapiro}
H^1\bigl(\Gamma_0^D(M),\cM_{\X_r}\bigr)\simeq H^1\Big(\Gamma_0^D(M),{\rm Coind}_{\Gamma_r}^{\Gamma_0^D(M)}(H\otimes \Z_p)\Big)\simeq H^1(\Gamma_r,H\otimes \Z_p).
\end{equation}
Finally, we obtain
\[ H^1\bigl(\Gamma_0^D(M),\cM_\X\bigr)\simeq H^1\bigl(\Gamma_0^D(M),\invlim_r\cM_{\X_r}\bigr)\simeq\invlim_r H^1\bigl(\Gamma_0^D(M),\cM_{\X_r}\bigr)\simeq\invlim_r H^1(\Gamma_r,H\otimes \Z_p) \]
where the first isomorphism follows from \eqref{eq-D}, the second from \cite[Corollary 2.3.5]{NSW} and the fact that $\Gamma_0^D(M)$ is finitely generated, and the third from \eqref{equation-shapiro}. \end{proof}

Now we prove a result which obviously implies Theorem \ref{prop-lifting}.

\begin{proposition} \label{com}
The map
\[ \pi_*:H^1\bigl(\Gamma_0^D(M),\cM_\X\bigr)\longrightarrow H^1\bigl(\Gamma_0^D(M),\cM\bigr) \]
is surjective.
\end{proposition}

\begin{proof} A simple computation shows that there is a commutative square
\[ \xymatrix@C=40pt@R=35pt{H^1\bigl(\Gamma_0^D(M),\cM_\X\bigr)\ar[r]^-{\pi_\ast}\ar[d]^-\simeq & H^1\bigl(\Gamma_0^D(M),\cM\bigr)\ar[d]^-\simeq\\
                      \invlim_rH^1(\Gamma_r,H\otimes \Z_p)\ar[r]^-{\invlim\mathrm{cor}_r} & \invlim_rH^1\bigl(\Gamma^D_0(p^rM),H\otimes \Z_p\bigr)} \]
in which
\[ \mathrm{cor}_r:=\mathrm{cor}_{\Gamma_r}^{\Gamma_0^D(p^rM)}:H^1(\Gamma_r,H\otimes \Z_p)\longrightarrow H^1\bigl(\Gamma_0^D(p^rM),H\otimes \Z_p\bigr) \]
is the corestriction (or transfer) map as defined, e.g., in \cite[Ch. III, \S 9]{Br} and the vertical isomorphisms are those in Proposition \ref{isos}. By Poincar\'e duality, and because $H$ is a free abelian group endowed with the trivial action of $\Gamma$, one has
\[ H^1(\Gamma_r,H\otimes \Z_p)\simeq H^1(\Gamma_r,\Z_p)\otimes H\simeq{\rm Ta}_p(J_r)\otimes H. \]
Here $J_r$ stands for the Jacobian variety of $X_r=\Gamma_r\backslash\cl H$ and ${\rm Ta}_p(J_r)$ is the $p$-adic Tate module of $J_r$. The restriction maps
\[ H^1(\Gamma_r,H\otimes \Z_p)\longrightarrow H^1(\Gamma_{r-1},H\otimes \Z_p) \]
turn out to be induced by the canonical maps between Tate modules
\[ {\rm Ta}_p(J_r)\longrightarrow{\rm Ta}_p(J_{r-1}) \]
arising from the universal property of Albanese varieties. Let us introduce the projective limit
\[ {\rm Ta}_p(J_\infty):=\invlim_r{\rm Ta}_p(J_r). \]
The diamond operators act on ${\rm Ta}_p(J_r)$ and induce an action of $1+p\Z_p$ on ${\rm Ta}_p(J_\infty)$. In this way the limit ${\rm Ta}_p(J_\infty)$ becomes a module over the Iwasawa algebra $\Lambda:=\Z_p[\![1+p\Z_p]\!]$.

Similarly, for all $r\geq1$ there is an isomorphism
\[ H^1\bigl(\Gamma_0^D(p^rM),H\otimes \Z_p\bigr)\simeq{\rm Ta}_p\bigl(J_0^D(p^rM)\bigr)\otimes H, \]
and we can again form the projective limit ${\rm Ta}_p\bigl(J_0^D(p^\infty M)\bigr):=\invlim{\rm Ta}_p\bigl(J_0^D(p^rM)\bigr)$.

If $I_\Lambda$ is the augmentation ideal of $\Lambda$ then the map
\[ \invlim_r H^1(\Gamma_r,H\otimes \Z_p)\xrightarrow{\invlim\mathrm{cor}_r}\invlim_r H^1\bigl(\Gamma^D_0(p^rM),H\otimes \Z_p\bigr) \]
corresponds, via the above isomorphisms, to the map
\begin{equation} \label{surj-lim-eq}
{\rm Ta}_p(J_\infty)\otimes H\longrightarrow\bigl({\rm Ta}_p(J_\infty)/I_\Lambda\cdot{\rm Ta}_p(J_\infty)\bigr)\otimes H\simeq{\rm Ta}_p\bigl(J_0^D(p^\infty M)\bigr)\otimes H
\end{equation}
because $J^D_0(p^rM)$ is precisely the quotient of $J_r$ by the action of the diamond operators. The map \eqref{surj-lim-eq} is visibly surjective, and the proposition is proved. \end{proof}

\subsection{Splitting cocycles}

Recall the homomorphism $\Phi:H_2(\Gamma,\Z)\rightarrow T(\C_p)$ from Section \ref{periods}, whose image is contained in $T(\Q_p)$ and is, by definition, the lattice $L$. Since $T(\C_p)$ is divisible, by the universal coefficient theorem there is a natural isomorphism
\[ H^2\bigl(\Gamma,T(\C_p)\bigr)\simeq\Hom\bigl(H_2(\Gamma,\Z),T(\C_p)\bigr), \]
and $\Phi$ defines in this way an element $\boldsymbol{d}\in H^2\bigl(\Gamma,T(\C_p)\bigr)$. By construction, the image in
\[ H^2\bigl(\Gamma,T(\C_p)/L\bigr)\simeq\Hom\bigl(H_2(\Gamma,\Z),T(\C_p)/L\bigr) \]
of the class $\boldsymbol{d}$ is trivial, and $L$ is the smallest subgroup of $T(\Q_p)$ with this property. Fix a point $\tau\in K_p-\Q_p$, i.e., a $K_p$-rational point on $\mathcal H_p$. Independently of this choice, the class $\boldsymbol{d}$ can be represented by the $2$-cocycle $d\in Z^2\bigl(\Gamma,T(K_p)\bigr)$ given by
\begin{equation} \label{desc-d}
d_{\gamma_1,\gamma_2}:=\mint_{\PP^1(\Q_p)}\frac{t-\gamma^{-1}_{1}(\tau)}{t-\tau}d\mu_{\gamma_2}(t),
\end{equation}
where $\mu$ is a cocycle in $Z^1\bigl(\Gamma,\cM_0(H)\bigr)$ representing $\boldsymbol\mu$, which we fix for the rest of this section.

Set $H_p:=H\otimes K_p$ for the rest of the article and consider the map
\[ \beta_{\mathcal L}:T(K_p)\longrightarrow H_p,\qquad h\otimes k\longmapsto h\otimes\log_p(k)-\mathcal L_p^D\cdot h\otimes\ord_p(k) \]
and the $2$-cocycle $\beta_{\mathcal L}\circ d\in Z^2(\Gamma,H_p)$, whose image in $H^2(\Gamma,H_p)$ we denote by $\boldsymbol{d}_{\mathcal L}$. Then $\beta_{\mathcal L}(L)$ is the smallest subgroup of $H_p$ such that the image of $\boldsymbol{d}_{\mathcal L}$ in $H^2\bigl(\Gamma,H_p/\beta_{\mathcal L}(L)\bigr)$ is trivial.

Theorem \ref{prop-I} is a direct consequence of the following result.

\begin{theorem} \label{vanish}
The cohomology class $\boldsymbol{d}_{\mathcal L}\in H^2(\Gamma,H_p)$ is trivial.
\end{theorem}

This is the statement that we will prove in various steps in the remaining subsections. To begin with, as in \S \ref{lattice}, we pick the base point $\tau\in K_p-\Q_p$ appearing in \eqref{desc-d} in such a way that $\mathrm{red}(\tau)=v_\ast$. In order to show that $\boldsymbol{d}_{\mathcal L}$ is trivial in $H^2(\G,H_p)$, we shall first prove that it splits when restricted to the subgroup $\Ga$.

With obvious notation, the first observation is that
\begin{equation} \label{d-log-eq}
(\boldsymbol{d}_{\mathcal L})_{|\Ga}=\log_p(\boldsymbol{d})_{|\Ga}.
\end{equation}
In fact, since $\mathrm{red}(\tau)=v_\ast$ and $\g_1$ lies in the stabilizer of this vertex, we have
\[ \mathrm{red}\bigl(\g_1^{-1}(\tau)\bigr)=\g^{-1}_1\bigl(\mathrm{red}(\tau)\bigr)=v_\ast, \]
thanks to the $\GL_2(\Q_p)$-invariance of the reduction map. Thus the geodesic joining $\mathrm{red}(\tau)$ with $\mathrm{red}\bigl(\g^{-1}_1(\tau)\bigr)$ is trivial and Proposition \ref{eq-II} asserts that
\[ \ord_p\bigg(\mint\frac{t-\gamma^{-1}_1(\tau)}{t-\tau}d\mu_{H,\gamma_2}^{\mathcal Y}(t)\bigg)=0. \]
Since $\ord_p\circ t_r=t_r\circ\ord_p$, it follows from \eqref{Heckepairing} and \eqref{desc-d} that $\ord_p(\boldsymbol{d})$ vanishes on $\Ga$, whence equality \eqref{d-log-eq}.

Now let $\boldsymbol{\tilde\mu}\in H^1\bigl(\Ga,\cM_\X\bigr)$ be as in Theorem \ref{prop-lifting} and choose $\tilde\mu\in Z^1\bigl(\Gamma_0^D(M),\cM_\X\bigr)$ representing $\boldsymbol{\tilde\mu}$. We claim that, at the cost of replacing it by a cohomologous cocycle, $\tilde\mu$ can be chosen such that $\pi_*(\tilde\mu)_\g=\mu_\g$ for all $\g\in\Gamma_0^D(M)$. For this, notice that it suffices to prove that the push-forward map $\pi_*:\cM_\X\rightarrow\cM$ is surjective. This can be shown, for example, using arguments borrowed from the proof of Proposition \ref{isos}, the crucial facts being that a measure in $\cM$ can be identified with a compatible sequence of maps $\Gamma_0^D(M)/\Gamma_0^D(Mp^r)\rightarrow H\otimes\Z_p$ for integers $r\geq1$ and that we have a canonical projection $\Gamma_0^D(M)/\Gamma_r\twoheadrightarrow\Gamma_0^D(M)/\Gamma_0^D(Mp^r)$ at our disposal, so that, after fixing compatible sets of representatives for $\Gamma_0^D(Mp^r)/\Gamma_r$, we can easily define a lifting to $\cM_\X$ of an element in $\cM$. Observe that these are cocycles with values in measures taking values in $H\otimes\Z_p$, which naturally embeds into $H_p$.

\begin{definition} \label{rho}
The $1$-cochain $\rho=\rho_\tau\in C^1\bigl(\Gamma_0^D(M),H_p\bigr)$ is defined as
\[ \rho_\g:=-\int_\X\log_p(x-\tau y)d\tilde\mu_\gamma(x,y). \]
\end{definition}

Note that $\rho$ depends both on the choice of $\tilde\mu$ and on the choice of $\tau$, but we shall drop any reference to either in order to lighten the notation.

\begin{proposition} \label{prop-splitting-I}
The $1$-cochain $\rho$ splits the $2$-cocycle $(d_{\mathcal L})_{|\Ga}=\log_p(d)_{|\Ga}$.
\end{proposition}

\begin{proof} We must show that
\[ \gamma_1\rho_{\gamma_2}+\rho_{\gamma_1}-\rho_{\gamma_1\gamma_2}=\log_p(d_{\gamma_1,\gamma_2}) \]
for all $\gamma_1,\gamma_2\in\Gamma_0^D(M)$. Since the action of $\Gamma_0^D(M)$ on $H\otimes K_p$ is trivial, one has
\[ \begin{split}
   \gamma_1\rho_{\gamma_2}+\rho_{\gamma_1}-\rho_{\gamma_1\gamma_2}&=-\int_\X\log_p(x-\tau y)d\bigl(\tilde\mu_{\gamma_1}+\tilde\mu_{\gamma_2}-\tilde\mu_{\gamma_1\gamma_2}\bigr)(x,y)\\[2mm]
   &=-\int_\X\log_p(x-\tau y)d\bigl(\tilde\mu_{\gamma_2}-{}^{\gamma_1}\tilde\mu_{\gamma_2}\bigr)(x,y).
   \end{split} \]
Thus if $\gamma_1=\smallmat abcd$ then
\[ \begin{split}
   \gamma_1\rho_{\gamma_2}+\rho_{\gamma_1}-\rho_{\gamma_1\gamma_2}&=-\int_\X\log_p(x-y\tau)d\bigl(\tilde\mu_{\gamma_2}-\gamma_1d\tilde\mu_{\gamma_2}\bigr)(x,y)\\[2mm]
   &=-\int_\X\log_p\left(\frac{x-\tau y}{ax+by-\tau(cx+dy)}\right)d\tilde\mu_{\gamma_2}(x,y).
   \end{split} \]
Now we argue as in \cite[Proposition 5.14]{Das}. Since the integrand above depends only on $x/y$, we deduce that
\[ \begin{split}
   \gamma_1\rho_{\gamma_2}+\rho_{\gamma_1}-\rho_{\gamma_1\gamma_2}&=\int_{\PP^1(\Q_p)}\log_p\left(\frac{at+b-(ct+d)\tau}{t-\tau}\right)d\mu_{\gamma_2}(t)\\[2mm]
   &=\int_{\PP^1(\Q_p)}\log_p\left(\frac{t-\gamma_1^{-1}(\tau)}{t-\tau}\right)d\mu_{\gamma_2}(t)-\int_{\PP^1(\Q_p)}\log_p\left(a-c\tau\right)d\mu_{\gamma_2}(t).
   \end{split} \]
Since $\mu_{\gamma_2}$ has total mass $0$, the last integral in the above expression vanishes, and the result follows from \eqref{desc-d}. \end{proof}

\subsection{Passing from $\Ga$ to $\hat\Gamma_0^D(M)$} \label{sec-shapiro}

Notations and some of the ideas in this subsection are borrowed from \cite[\S 1]{AS}. Let $Y$ be a locally compact and totally disconnected ($p$-adic) topological space endowed with a left action of $\M_2(\Z_p)$.

Recall from \S \ref{HZ} the elements $g_i=g_i(p)$ for $i=0,\dots,p-1$, which give rise to the decomposition of the double cosets associated with the Hecke operator $U_p$, and recall also the set of representatives
\[ \bigl\{\alpha_\infty:=1,\alpha_0:=\om_p^{-1}g_0,\dots,\alpha_{p-1}:=\om_p^{-1}g_{p-1}\bigr\} \]
for the cosets $\G_0^D(M)/\Gap$. Assume there exists a compact open subset $Y_\infty$ of $Y$ satisfying the following conditions:

\begin{itemize}
\item[(I)] $\g Y_\infty=Y_\infty$ for all $\g\in\Gamma_0^D(pM)$;
\item[(II)] if $Y_i:=\alpha_i\cdot Y_\infty$ for $i=0,\dots,p-1$ and $Y_{\mathrm{aff}}:=\coprod_{i=0}^{p-1}Y_i$ then $Y=Y_\infty\coprod Y_{\mathrm{aff}}$;
\item[(III)] $g_i\cdot Y_\infty\subset Y_\infty$ and $\coprod_{i=0}^{p-1}g_i\cdot Y_\infty=Y_\infty$;
\item[(IV)] $\om_p\cdot Y_{\mathrm{aff}}=Y_\infty$ and $\om_p\cdot Y_\infty=pY_{\mathrm{aff}}$, so that $\om_p\cdot Y=Y_\infty\coprod pY_{\mathrm{aff}}$.
\end{itemize}
Then it follows from (I) and (II) that
\[ \cM_Y\simeq{\rm Coind}^{\Ga}_{\Gap}(\cM_{Y_\infty}), \]
and Shapiro's lemma produces an isomorphism
\[ \mathscr{S}:H^1\bigl(\G_0^D(M),\cM_Y\bigr)\overset\simeq\longrightarrow H^1\bigl(\Gap,\cM_{Y_\infty}\bigr). \]
Conditions (III) and (IV) on $Y_\infty$ ensure that the Hecke operator $U_p$ is well defined and well behaved on $H^1\bigl(\Gamma_0^D(pM),\cM_{Y_\infty}\bigr)$. In the spirit of \cite[Lemma 1.1.4]{AS}, we transport the operator $U_p$ to an operator on $H^1\bigl(\G_0^D(M),\cM_Y\bigr)$ by means of the isomorphism $\mathscr{S}$. Namely, define
\begin{equation} \label{Up}
{\boldsymbol U}_p:=\mathscr{S}^{-1}U_p\mathscr{S}.
\end{equation}
The same argument applied to $\hat\G_0^D(M)$ in place of $\G_0^D(M)$ shows the existence of an isomorphism
\[ \hat{\mathscr{S}}:H^1\bigl(\hat\G_0^D(M),\cM_{\om_p\cdot Y}\bigr)\overset\simeq\longrightarrow H^1\bigl(\G_0^D(pM),\cM_{\om_p\cdot Y_\infty}\bigr). \]

\begin{lemma} \label{lemma-wp-up}
For every $\boldsymbol{\nu}\in H^1\bigl(\Gamma_0^D(M),\cM_Y\bigr)$ one has
\begin{itemize}
\item[(i)] $W_p^{-1}U_p\bigl(\res_{\Gamma_0^D(pM)}(\boldsymbol{\nu}_{|Y_{\infty}})\bigr)=\res_{\Gamma_0^D(pM)}(\boldsymbol{\nu}_{|Y_{\mathrm{aff}}})$;
\item[(ii)] $U_p^2\bigl(\res_{\Gamma_0^D(pM)}\boldsymbol{\nu}_{|Y_\infty}\bigr)=\bigl(W_pU_p\res_{\Gamma_0^D(pM)}\boldsymbol{\nu}\bigr)_{|Y_\infty}$;
\item[(iii)] $\bigl(W_pU_p\res_{\Gamma_0^D(pM)}\boldsymbol{\nu}\bigr)_{|pY_{\mathrm{aff}}}=W_pU_p\bigl(\res_{\Gamma_0^D(pM)}\boldsymbol{\nu}_{|Y_\infty}\bigr)=\bigl(W_p^2\res_{\G_0^D(pM)}(\boldsymbol{\nu})\bigr)_{|pY_{\mathrm{aff}}}$.
\end{itemize}
Moreover, for every $\boldsymbol{\nu}\in H^1\bigl(\hat\Gamma_0^D(M),\cM_{\om_p\cdot Y}\bigr)$ one has
\begin{itemize}
\item[(iv)] $U_pW_p^{-1}\bigl(\res_{\Gamma_0^D(pM)}(\boldsymbol{\nu}_{|pY_{\mathrm{aff}}})\bigr)=\res_{\Gamma_0^D(pM)}(\boldsymbol{\nu}_{|Y_\infty})$.
\end{itemize}
\end{lemma}

\begin{proof} Let us show (i) first, the remaining statements being applications of or variations on it. Let $\boldsymbol{\nu}\in H^1\bigl(\Gamma_0^D(M),\cM_Y\bigr)$ and fix a representative $\nu$ of $\boldsymbol\nu$ in $Z^1\bigl(\Gamma_0^D(M),\cM_Y\bigr)$; set $n:=\res_{\Gap}(\nu_{|Y_\infty})$. An easy formal calculation shows that for all $\g\in\Gap$ the equality
\begin{equation} \label{eq-wp-up}
\alpha_i\cdot\nu_{\alpha_i^{-1}\g\alpha_{j(i)}}=\nu_\g-\nu_{\alpha_i}+\g\cdot \nu_{\alpha_{j(i)}}
\end{equation}
holds in $Z^1\bigl(\Gamma_0^D(M),\cM_Y\bigr)$. Here $i\mapsto j(i)$ is the permutation of indices such that $\alpha_i^{-1}\g\alpha_{j(i)}\in\Gap$.

Since $\alpha_i\cdot n_{\alpha_i^{-1}\g\alpha_{j(i)}}$ is supported on $Y_i$, we deduce that
\[ \alpha_i\cdot n_{\alpha_i^{-1}\g\alpha_{j(i)}}=\nu_{\g|Y_i}-\nu_{\alpha_i|Y_i}+(\g\nu_{\alpha_{j(i)}})_{|Y_i}. \]
Note, however, that $\g(\nu_{\alpha_{j(i)}|Y_{j(i)}})=(\g\nu_{\alpha_{j(i)}})_{|Y_i}$; the reason is that, since $\g$ belongs to $\alpha_i\cdot\Gap\alpha_{j(i)}^{-1}$, both measures are supported at $Y_i$ and are in fact the restriction of $\nu _{\alpha_{j(i)}}$ to this compact open subset of $Y$.

Setting $m:=\sum_{i=0}^{p-1}\nu_{\alpha_i|Y_i}\in \cM_{Y_{\mathrm{aff}}}$, equality \eqref{eq-wp-up} shows that
\[ \sum_{i=0}^{p-1}\alpha_i\nu_{\alpha_i^{-1}\g\alpha_{j(i)}}=\sum_{i=0}^{p-1}\nu_{\g|Y_i}+\g m-m. \]
Since $\alpha_i=\om_p^{-1}g_i$, by definition the first term is $W_p^{-1}U_p(n)$. On the other hand, the second term equals $\res_{\Gap}(\nu_{|Y_{\mathrm{aff}}})$ in $H^1\bigl(\Gap,\cM_{Y_{\mathrm{aff}}}\bigr)$, and (i) is proved.

For (iii), it suffices to show that
\[ W_pU_p\bigl(\res_{\Gamma_0^D(pM)}\boldsymbol{\nu}_{|Y_\infty}\bigr)=W_p^2\res_{\Gamma_0^D(pM)}\bigl(\boldsymbol{\nu}_{|Y_{\mathrm{aff}}}\bigr), \]
and this is is deduced from (i) upon applying $W_p^2$.

Part (iv) follows from (i) by taking into account that for every compact open subset $U$ of $Y$ the map $W_p$ induces an isomorphism
\[ W_p:H^1\bigl(\Gamma_0^D(M),\cM_U\bigr)\overset\simeq\longrightarrow H^1\bigl(\hat\Gamma_0^D(M),\cM_{\om_p\cdot U}\bigr). \]
Finally, to check (ii) it is again enough to prove that
\[ U_p^2\bigl(\res_{\Gamma_0^D(pM)}\boldsymbol{\nu}_{|Y_\infty})=W_pU_p\res_{\Gamma_0^D(pM)}\bigl(\boldsymbol{\nu}_{|Y_{\mathrm{aff}}}\bigr), \]
which follows by applying (iv) to $W_p{\boldsymbol U}_p\boldsymbol{\nu}$. \end{proof}

Since $\om_p^2=p\cdot g_p$ for some $g_p\in\Gap$, the map $W_p^2$ sends an element $\nu\in Z^1\bigl(\Gamma_0^D(M),\cM_U\bigr)$ to the cocycle $\g\mapsto p\cdot g_p\nu_{g_p^{-1}\g g_p}$. A straightforward calculation then shows that
\[ g_p\nu_{g_p^{-1}\g g_p}=\nu_{\g}+\g\nu_{g_p}-\nu_{g_p}. \]
Thus, since the map $\g\mapsto\g\nu_{g_p}-\nu_{g_p}$ is a coboundary, the equality
\begin{equation} \label{integral-I}
W_p^2\boldsymbol{\nu}=p\cdot\boldsymbol{\nu}
\end{equation} holds
in $H^1\bigl(\Gamma_0^D(pM),\cM_U\bigr)$ for every compact open subset $U$ of $Y$.

\subsection{Splitting on $\hat\Gamma_0^D(M)$}

Define $\boldsymbol{\hat{\tilde\mu}}:=W_p{\boldsymbol U}_p\boldsymbol{\tilde\mu}\in H^1\bigl(\hat\Gamma_0^D(M),\cM_{\om_p\X}\bigr)$ and let $\hat{\tilde\mu}$ be a $1$-cocycle representing $\boldsymbol{\hat{\tilde\mu}}$. As above, we can define a $\Z_p^\times$-bundle
\[ \hat\pi:\om_p\X\longrightarrow\PP^1(\Q_p),\qquad(x,y)\longmapsto x/y \]
which induces a map $\hat\pi_*$ on cohomology.

\begin{lemma} \label{lemma-splitting-II}
$\hat\pi_*\bigl(\boldsymbol{\hat{\tilde\mu}}\bigr)=\res_{\hat\G_0^D(M)}\,(\boldsymbol{\mu})$.
\end{lemma}

\begin{proof} It is immediate to check that $W_p{\boldsymbol U}_p\boldsymbol{\tilde\mu}$ is a lift of $W_p{\boldsymbol U}_p\cdot\res_{\Ga}(\boldsymbol{\mu})$, that is
\[ \hat\pi_*\bigl(\boldsymbol{\hat{\tilde\mu}}\bigr)=W_p{\boldsymbol U}_p\cdot\res_{\Ga}(\boldsymbol{\mu}). \]
According to \eqref{Up}, one has ${\boldsymbol U}_p=\mathscr{S}^{-1}\cdot U_p\cdot\mathscr{S}$. Since $\Gap$ is a subgroup of finite index in $\Ga$, \cite[Lemma 1.1.4]{AS} ensures that Shapiro's isomorphism $\mathscr{S}$ commutes with the action of $W_p$. More precisely, we have $\hat{\mathscr{S}}^{-1}\cdot W_p=W_p\cdot\mathscr{S}^{-1}$, hence we must show that
\[ \hat{\mathscr{S}}^{-1}\cdot W_p\cdot U_p\cdot\mathscr{S}\bigl(\res_{\Ga}(\boldsymbol{\mu})\bigr)=\res_{\hat\G_0^D(M)}\,(\boldsymbol{\mu}). \]
Thanks to part (iii) of Lemma \ref{lemma-wp-up} applied to $\boldsymbol{\nu}:=\res_{\Ga}(\boldsymbol{\mu})$, we have
\[ W_p\cdot U_p\cdot\mathscr{S}(\boldsymbol{\nu})=W_p^2\bigl(\res_{\Gap}(\boldsymbol{\nu}_{|\Z_p})\bigr) \]
in $H^1\bigl(\Gap,p\Z_p\bigr)$. Noting that $\PP^1(\Q_p)_{\mathrm{aff}}=\Z_p$, it thus suffices to show that
\[ \hat{\mathscr{S}}\bigl(\res_{\hat{\Gamma}_0^D(M)}(\boldsymbol{\mu})\bigr)=W_p^2\bigl(\res_{\Gap}(\boldsymbol{\mu}_{|\Z_p})\bigr). \]
On the left hand side, $\hat{\mathscr{S}}\bigl(\res_{\hat{\Gamma}_0^D(M)}(\boldsymbol{\mu})\bigr)$ is equal, by definition, to $\res_{\Gap}\bigl(\boldsymbol{\mu}_{|\Z_p}\bigr)$, since $\om_p\bigl(\PP^1(\Q_p)-\Z_p\bigr)=\Z_p$. On the right hand side, since the action of $\GL_2(\Q_p)$ on $\PP^1(\Q_p)$ factors through $\PGL_2(\Q_p)$, we can argue as in \eqref{integral-I} and obtain that $W_p^2(m)=\res_{\Gap}\bigl(\boldsymbol{\mu}_{|\Z_p}\bigr)$ for $m=\res_{\Gap}(\boldsymbol{\nu}_{|\Z_p})$, as we wished to show. \end{proof}

Thanks to Lemma \ref{lemma-wp-up} and equality \eqref{integral-I}, for all $\hat\g\in\hat\G_0^D(pM)$ we can write
\begin{equation} \label{eq-new}
\hat{\tilde\mu}_{\hat\g}=U_p^2\tilde\mu_{\hat\g}+\hat\g m_1-m_1\quad\text{on $\X_\infty$},\qquad\hat{\tilde\mu}_{\hat\g}=p\tilde\mu_{\hat\g}+\hat\g m_2-m_2\quad\text{on $p\X_{\rm aff}$}
\end{equation}
with $m_1\in\cM_{\X_\infty}$ and $m_2\in\cM_{\X_{\rm aff}}$. The same argument as in the proof of Lemma \ref{new} shows that the cocycle $\mu^{(p)}_{\mathbb H}\in
Z^1\bigl(\G,\cF_{\rm har}(\mathbb H)\bigr)$ given by
\[ \mu^{(p)}_{\mathbb H,\g}(e):=\sum_i\pi_\mathbb H\bigl(\bigl[t_i(g_{\g,e})\bigr]\bigr) \]
for $\g\in\G$ and $e\in\E^+$ (where the functions $t_i$ are relative to the Hecke operator $U_p$) satisfies the equation
\begin{equation} \label{U_p-mu-eq}
U_p\bigl(\mu_{\mathbb H}^{\mathcal Y}\bigr)=\mu^{(p)}_{\mathbb H}+b
\end{equation}
for some $b\in\ker(\varrho)\subset Z^1\bigl(\G,\HC(\mathbb H)\bigr)$. Since $\ker(\varrho)$ is Eisenstein (cf. the proof of Lemma \ref{Eis}), applying $t_r$ to \eqref{U_p-mu-eq} and recalling that the action of $U_p$ on $\mathbb H$ is by $\pm1$ yields the equality $U_p(\mu_\g)=\pm\mu_\g$, from which we finally deduce that
\[ U_p^2(\res_{\G_0^D(pM)}\mu)_\g=U_p^2(\mu_\g)=\mu_\g. \]
Furthermore, it is clear that $p\mu_\g=\mu_\g$ for all $\g\in\G_0^D(pM)$, because the action of $\GL_2(\Q_p)$ on $\PP^1(\Q_p)$ factors through $\PGL_2(\Q_p)$. Thus we find from \eqref{eq-new} and Lemma \ref{lemma-splitting-II} that $\hat\pi_\ast(m_1)$ and $\hat\pi_\ast(m_2)$ are $\G_0^D(pM)$-invariant $H_p$-valued measures on $\PP^1(\Q_p)$ and $\Z_p$, respectively. One easily shows that the groups of such measures are trivial; the reason for this is that the $\G_0^D(pM)$-invariance would otherwise contradict the fact that measures have to be $p$-adically bounded. Thus $\hat\pi_\ast(m_1)=0$ on $\PP^1(\Q_p)-\Z_p$ and $\hat\pi_\ast(m_2)=0$ on $\Z_p$.

Now define the $1$-cochain $\hat\rho\in C^1\bigl(\hat\Gamma_0^D(M),H_p\bigr)$ by the rule
\[ \begin{split}
   \hat\rho_{\hat\gamma}:=&-\int_{\om_p\X}\log_p(x-\tau y)d\hat{\tilde\mu}_{\hat\g}(x,y)+\int_{\X_\infty}\log_p(x-\tau y)d(\hat\g m_1-m_1)(x,y)\\
                          &+\int_{p\X_{\rm aff}}\log_p(x-\tau y)d(\hat\g m_2-m_2)(x,y)
   \end{split} \]
for all $\hat\gamma\in\hat\Gamma_0^D(M)$. As above, this cochain depends on $\tau$ and on the choices made for the representatives of the cohomology classes. Nevertheless, we can prove

\begin{proposition} \label{prop-splitting-II}
The $1$-cochain $\hat\rho$ splits the $2$-cocycle $\log_p(d)_{|\hGa}$.
\end{proposition}

\begin{proof} As in the proof of Proposition \ref{prop-splitting-I}, if $\hat\gamma_1=\smallmat abcd$ then $\hat{\g}_1\hat{\rho}_{\hat{\g}_2}+\hat{\rho}_{\hat{\g}_1}-\hat{\rho}_{\hat{\g}_1\hat{\g}_2}$ is the sum of the three integrals
\[ A:=-\int_{\om_p\X}\log_p\left(\frac{x-\tau y}{ax+by-\tau(cx+dy)}\right)d\hat{\tilde\mu}_{\hat\gamma_2}(x,y), \]
\[ B:=\int_{\X_\infty}\log_p\left(\frac{x-\tau y}{ax+by-\tau(cx+dy)}\right)d(\hat\g_2m_1-m_1)(x,y), \]
\[ C:=\int_{p\X_{\rm aff}}\log_p\left(\frac{x-\tau y}{ax+by-\tau(cx+dy)}\right)d(\hat\g_2m_2-m_2)(x,y). \]
Since $B$ and $C$ depend only on $x/y$, from the vanishing of $\pi_*(m_1)$ and $\pi_*(m_2)$ we deduce that $B=C=0$. As in the proof of Proposition \ref{prop-splitting-I}, the claim follows from Lemma \ref{lemma-splitting-II}. \end{proof}

\subsection{Proof of Theorem \ref{vanish}}

Set $U:=\PP^1(\Q_p)-\Z_p$ and write $\mu$ for a cocycle in $Z^1(\Gamma,\mathcal M)$ representing $\boldsymbol\mu$. Since every $\g\in\Gap$ leaves $U$ invariant, the cochain $\mu_U$ given by the rule
\[ \mu_U(\g):=\bigl(\mu_{|\Gamma_0^D(pM)}\bigr)_{\g}(U) \]
is independent of the choice of $\mu$ and belongs to $Z^1\bigl(\Gap,H_p\bigr)$. Below, by
\[ \cl L_p^D\cdot\mu_U \]
we obviously mean the cocycle $\g\mapsto\cl L_p^D\cdot\mu_U(\g)$, with $\cL_p^D$ acting on $H$ as usual.

\begin{lemma} \label{lemma-splitting-III}
The equality $(\rho-\hat\rho)_{|\Gap}=\cl L_p^D\cdot\mu_U$ holds in $Z^1\bigl(\Gap,H_p\bigr)$.
\end{lemma}

\begin{proof} Let $\g\in\Gamma_0^D(pM)$. Using the decompositions $\X=\X_\infty\coprod\X_{\mathrm{aff}}$ and $\om_p\X=\X_\infty\coprod p\X_{\mathrm{aff}}$ we split the difference
\[ \begin{split}\rho_\g-\hat\rho_\g=&-\int_\X\log_p(x-\tau y)d\tilde\mu_{\gamma}+\int_{\om_p\X}\log_p(x-\tau y)d\hat{\tilde\mu}_\gamma\\
                                    &-\int_{\X_\infty}\log_p(x-\tau y)d(\g m_1-m_1)-\int_{p\X_{\rm aff}}\log_p(x-\tau y)d(\g m_2-m_2)
   \end{split} \]
into the sum of $A$ and $B$ with
\[ A:=-\int_{\X_\infty}\log_p(x-\tau y)d\tilde\mu_\gamma+\int_{\X_\infty}\log_p(x-\tau y)d\hat{\tilde\mu}_\gamma-\int_{\X_\infty}\log_p(x-\tau y)d(\g m_1-m_1) \]
and
\[ B:=-\int_{\X_{\mathrm{aff}}}\log_p(x-\tau y)d\tilde\mu_\gamma+\int_{p\X_{\mathrm{aff}}}\log_p(x-\tau y)d\hat{\tilde\mu}_\gamma-\int_{p\X_{\rm aff}}\log_p(x-\tau y)d(\g m_2-m_2). \]
By formulas \eqref{eq-new}, one has
\[ A=-\int_{\X_\infty}\log_p(x-\tau y)\bigl(1-U_p^2\bigr)d\tilde\mu_\g. \]
As for $B$, since $\log_p(p)=0$, using again \eqref{eq-new} we can write
\[ \begin{split}
   \int_{p\X_{\mathrm{aff}}}\log_p(x-\tau y)d\bigl(W_p^2\tilde\mu\bigr)_\gamma &=\int_{\X_{\mathrm{aff}}}\log_p(px-\tau py)d{\tilde\mu}_\gamma+\int_{p\X_{\rm aff}}\log_p(x-\tau y)d(\g m_2-m_2)\\
                                                                               &=\int_{\X_{\mathrm{aff}}}\log_p(x-\tau y)d\tilde\mu_\g+\int_{p\X_{\rm aff}}\log_p(x-\tau y)d(\g m_2-m_2),
   \end{split} \]
whence $B=0$. Therefore
\begin{equation} \label{rho-int-eq}
(\rho-\hat\rho)_\g=-\int_{\X_\infty}\log_p(x-\tau y)\bigl(1-U_p^2\bigr)d\tilde\mu_\g.
\end{equation}
By \cite[Lemma 5.16]{Das} (once one makes the obvious modifications in the notation; namely, replace $m$ by $\g$ and $-\psi(m)$ by $\mu_{\g}(U)$), the integral in \eqref{rho-int-eq} is equal to $\cl L_p^D\cdot\mu_\g(U)$, and the proof is complete. \end{proof}

Since $\G=\Ga\ast_{\Gap}\hat{\Gamma}_0^D(M)$ by \eqref{amalgam}, the Mayer--Vietoris long exact sequence for amalgamated products of groups (cf. \cite[Theorem 2.3]{Sw}) yields an exact sequence
\[ H^1\bigl(\Gap,H_p\bigr)\overset\Delta\longrightarrow H^2(\G,H_p)\rightarrow H^2\bigl(\Ga,H_p\bigr)\oplus H^2\bigl(\hat{\Gamma}_0^D(M),H_p\bigr)\rightarrow
H^2\bigl(\Gamma_0^D(pM),H_p\bigr) \]
which, by means of the identifications provided by Shapiro's lemma, can also be regarded as the long exact sequence in cohomology associated with the short exact sequence of $\G$-modules
\[ 0\longrightarrow H_p\longrightarrow\cF(\V,H_p)\overset{F}{\longrightarrow}\cF_0(\E,H_p)\longrightarrow0 \]
with $F(f)(e):=f\bigl(t(e)\bigr)-f\bigl(s(e)\bigr)$ for all $e\in\E$. Observe that this exact sequence is nothing other than the dual of \eqref{exseq}.

Let $\boldsymbol{\rho-\hat\rho}$ denote the class of the cocycle $(\rho-\hat\rho)_{|\Gap}$ in $H^1\bigl(\Gap,H_p\bigr)$. The last ingredient we need is the following

\begin{proposition} \label{prop-splitting-III}
$\Delta(\boldsymbol{\rho-\hat\rho})=\cl L_p^D\cdot\ord_p(\boldsymbol{d})$.
\end{proposition}

\begin{proof} Writing $\boldsymbol\mu_U$ for the class of the cocycle $\mu_U$ in $H^1\bigl(\Gamma_0^D(pM),H_p\bigr)$, by Lemma \ref{lemma-splitting-III} it is enough to prove that
\[ \Delta(\boldsymbol{\mu}_U)=\ord_p(\boldsymbol{d}) \]
in $H^2(\G,H_p)$. This equality, which is the counterpart of \cite[Equation (22)]{Gr}, follows by combining Proposition \ref{eq-II}, the commutativity relation \eqref{commu} and the explicit description of the map $\Delta$. \end{proof}

Now we can prove Theorem \ref{vanish}, which implies Theorem \ref{prop-I}.

\begin{proof}[Proof of Theorem \ref{vanish}.]  The combination of Propositions \ref{prop-splitting-I} and \ref{prop-splitting-II} ensures that $\log_p(\boldsymbol{d})$ lies in the image of $\Delta$; in fact, it follows from the definition of the maps involved in the above Mayer--Vietoris sequence that
\[ \log_p(\boldsymbol{d})=\Delta(\boldsymbol{\rho-\hat\rho}), \]
because $\rho$ and $\hat\rho$ split $\log_p(\boldsymbol{d})_{|\Ga}$ and $\log_p(\boldsymbol{d})_{|\hGa}$, respectively. Proposition \ref{prop-splitting-III} then asserts that
\[ \log_p(\boldsymbol{d})=\cl L_p^D\cdot\ord_p(\boldsymbol{d}), \]
hence $\boldsymbol{d}_{\mathcal L}$ is trivial in $H^2(\G, H_p)$. \end{proof}

\subsection{Proof of a conjecture of M. Greenberg} \label{greenberg-subsec}

As an application of Theorem \ref{prop-I}, we give a proof of the conjecture formulated by M. Greenberg in \cite[Conjecture 2]{Gr} in the special case where the totally real field appearing in \cite{Gr} is $\Q$.

To state this result, let $E_{/\Q}$ be an elliptic curve of conductor $N=pMD$ and let $K$ be a real quadratic field in which the primes dividing $M$ split and the primes dividing $pD$ are inert. In particular, the completion $K_p$ of $K$ at the unique prime above $p$ is the unramified quadratic extension of $\Q_p$, so this notation is consistent with the one used in the rest of the paper. Observe that $E$ acquires split multiplicative reduction over $K_p$, write $q_E\in p\Z_p$ for Tate's $p$-adic period of $E$ and let $\langle q_E\rangle$ be the lattice in $K^\times_p$ generated by $q_E$.

Now, as in \cite[\S 3.4]{Gr}, choose a sign $\epsilon\in\{\pm1\}$ and set $H_E:=H_1(E,\Z)^\epsilon$. With the notation used in the previous sections of this paper, there are natural Hecke-equivariant surjections
\[ H_1\bigl(\Gap,\Z\bigr)\longrightarrow H\overset{\pi_E}\longrightarrow H_E. \]
One can attach to $E$ the measure-valued cohomology class $\boldsymbol{\mu}_E:=\boldsymbol{\mu}_{H_E}\in H^1\bigl(\G,\cM_0(H_E)\bigr)$ introduced at the end of \S \ref{sec-def-meas}.

Fix an isomorphism $H_E\simeq\Z$. For every prime $\ell$ write $a_\ell(E)$ for the $\ell$-th Fourier coefficient in the $q$-expansion of the newform associated with $E$ by modularity. Thanks to Proposition \ref{HeckeEquiv} and Lemma \ref{new}, it is immediate to show that $\boldsymbol{\mu}_E$ spans the one-dimensional subspace of $H^1\bigl(\G,\cM_0(\Q)\bigr)$ on which the Hecke algebra $\mathcal H(p,M)$ acts via the map
\[ \lambda_E:\mathcal H(pM)\longrightarrow\Z \]
attached to $E$ such that
\[ \lambda_E(T_\ell):=a_\ell(E)\quad\text{if $\ell\nmid pDM$},\qquad\lambda_E(W_p):=a_p(E),\qquad\lambda_E(W_\infty):=\epsilon. \]
Hence we conclude that our measure-valued class is an explicit version of the one considered in \cite[\S 8, (17)]{Gr}. Recall from Sections \ref{raav} and \ref{periods} that there is a pairing
\[ \langle\,,\rangle_E:H_1(\Gamma,\cD)\times H^1\bigl(\Gamma,\cM_0(H_E)\bigr)\longrightarrow\C_p^\times\otimes H_E\simeq\C_p^\times \]
and a Hecke-equivariant integration map
\[ \int_E:H_1(\G,\cD)\longrightarrow\C_p^\times\]
which fits into the commutative triangle
\[ \xymatrix@C=30pt@R=30pt{H_1(\G,\cD)\ar[r]^-\int\ar[dr]_-{\int_E} & T(\C_p)\ar[d]^-{\mathrm{id}\otimes\pi_E}\\
             & \C_p^\times.} \]
Set
\[ \Phi_E:=\int_E\circ\;\partial:H_2(\G,\Z)\longrightarrow\C_p^\times \]
and let $L_E\subset\C_p^\times$ denote the image of $\Phi_E$. Arguing as in the proof of Theorem \ref{lattice}, or invoking \cite[Proposition 30]{Gr}, it follows that $L_E$ is a lattice in $K_p^\times$.

As in \cite[Definition 29]{Gr}, we say that two lattices $\Lambda_1$ and $\Lambda_2$ in $K_p^\times$ are \emph{homothetic} if $\Lambda_1\cap\Lambda_2$ has finite index in both $\Lambda_1$ and $\Lambda_2$.

The result we want to prove, which was originally proposed in \cite[Conjecture 2]{Gr}, is the following

\begin{theorem} \label{greenberg-prop}
The lattices $L_E$ and $\langle q_E\rangle$ are homothetic in $K_p^\times$.
\end{theorem}

\begin{proof} Multiplicity one ensures that the Tate elliptic curve $K_p^\times/\langle q_E\rangle$ is, up to isogeny, the unique quotient of $J_0^D(pM)^{\text{$p$-new}}$ on which the action of the Hecke operators $T_\ell$ for $\ell\nmid pDM$ and of the Atkin--Lehner involutions $W_p$ and $W_\infty$ factors through $\lambda_E$. Similarly, $K_p^\times\otimes H_E\simeq K_p^\times$ is the unique quotient of $K_p^\times\otimes H$ on which the action of these operators factors through $\lambda_E$.

Hence it follows from Theorem \ref{GreenbergConj2} that $K_p^\times/\langle q_E\rangle$ and $K_p^\times/L_E$ are isogenous over $K_p$, which amounts to saying that the lattices $L_E$ and $\langle q_E\rangle$ are homothetic in $K^\times_p$. \end{proof}

\begin{remarkwr}
If $f\in S_2(N)^{\text{$p$-new}}$ is a normalized $p$-new eigenform with not necessarily integral Fourier coefficients then Theorem \ref{greenberg-prop}, with the obvious modifications in the statement and in the proof, holds true as well.
\end{remarkwr}

\end{document}